%% file: R-1352-DIE-00.tex
 \documentclass[11pt,twoside,reqno,centertags]{amsart}

 \PII{} 
\copyrightinfo{}{}


\usepackage{amsmath,amsthm,amsfonts,amssymb}
\usepackage[mathscr]{euscript}
\usepackage[hidelinks]{hyperref} 
\usepackage{mathrsfs} 
\usepackage{graphicx}
\usepackage{color}
\usepackage{epsfig}

\usepackage[T1]{fontenc}
\usepackage{textcomp}
\usepackage{amsfonts}
\usepackage{amssymb,float,bm}
\usepackage{tikz} 
\usepackage{caption}
\usepackage{subcaption}
\usepackage{mathtools}
\usepackage[overload]{empheq}

\newcommand{\sech}{\text{sech\,}}
\newtheorem{thm}{Theorem}[section]

\newtheorem{lemma}[thm]{Lemma}

\theoremstyle{remark}

\newtheorem{theorem}{Theorem}[section]

\newtheorem{proposition}[theorem]{Proposition}
\newtheorem{corollary}[theorem]{Corollary}

\definecolor{azure}{rgb}{0.0, 0.5, 1.0} 
\definecolor{bleudefrance}{rgb}{0.16, 0.32, 0.75} 
\definecolor{ballblue}{rgb}{0.13, 0.67, 0.8} 

\numberwithin{equation}{section}   

\begin{document}
 
\title{ Flow kinematics for equatorial coupled surface and internal waves }
\thanks{Accepted for publication: February 2026.} 
\date{}
\maketitle     
 
\vspace{ -1\baselineskip}

{\small
\begin{center}
{\sc David Henry}\\
School of Mathematical Sciences, University College Cork, Cork, Ireland\\[10pt]
{\sc Rossen Ivanov}\\
School of Mathematics and Statistics, 
Technological University Dublin\\
Grangegorman Lower, D07 ADY7 Dublin, Ireland\\[10pt]
{\sc Gabriele Villari}\\
Dipartimento di Matematica e informatica ``Ulisse Dini" \\
Universit\`a degli Studi di Firenze\\
viale Morgagni, 67/A, 50137 Firenze, Italy \\[10pt]
 (Submitted by: Adrian Constantin)  
\end{center}
}

\numberwithin{equation}{section}
\allowdisplaybreaks

 \smallskip

 \begin{quote}
\footnotesize
{\bf Abstract.}  
We study the propagation of coupled surface and internal equatorial internal waves. 
A model of two vertically stratified fluid layers with different constant densities is employed. Taking Coriolis 
forces into account, we derive explicit solutions to
the linearized governing equations which 
assumes irrotational fluid motion in both layers
separately, and further obtain the dispersion
 relation which determines the phase speeds
of propagating surface and internal waves. 
We prove a result on solutions to the 
dispersion relations which greatly simplifies our 
subsequent analysis of the nonlinear dynamical
systems which describe the motion of the 
fluid in the upper layer. Phase portraits for all
possible streamlines in both fluid layers are 
presented, while furthermore a Lagrangian
description of the fluid flow is obtained, 
and the particle trajectories of the fluid particles
are determined.

\end{quote}

\section{Introduction}
Wave motion is a fascinating and highly perplexing field of scientific study, and this is no more evident than in the range of mathematical challenges it presents. One aspect of wave motion which has seen significant progress in recent years is in the theoretical understanding of the underlying flow which is induced by various wave motions. In a sense, this subject dates from the time of Stokes (mid 1800s, see \cite{Craik}) and his observations regarding the mean drift experienced by fluid particles as a wave propagates on the surface of water: the topic of Stokes' drift are still an active area of research with many open questions \cite{Bremer,Buhler}. Most studies of the underlying flow generated by wave motion concern the mean properties of the kinematics,  whereby  an averaging process is first necessarily employed in order to reduce the complexity of the overall problem. 

Determining the motion of individual fluid particles as a fluid body evolves due to wave propagation requires a delicate analysis of the underlying mathematical structure \cite{Abraskin,Bennett}. Mathematical advances enabled significant progress in this area in the past couple of decades  for waves on the surface of a single homogeneous layer of fluid (see \cite{Con-11,Con-15} for an overview). We note that, while this work extends to the fully nonlinear regime of exact solutions to the Euler equations which govern the wave motion (see again \cite{Con-11,Con-15} for an overview), there are still a number of simplifications which must be invoked, such as periodicity and, in the nonlinear setting, irrotationality.

Due to the technical complexities inherent in their analysis,  no such detailed studies had existed concerning the flow induced by internal wave motion on a multi-layered fluid body until the recent work \cite{HV-JDE,HV-AN}, which applied a phase-plane approach to the study of the nonlinear dynamical systems which govern the motion of coupled surface and internal wave solutions arising from the linearised governing equations. It is  assumed in  \cite{HV-JDE,HV-AN} that the  fluid layers are separately irrotational, although  the forthcoming works \cite{HIS-JDE,HIS-pre} show that the incorporation of vorticity in the analysis is also possible. We note that  these above cited works allow for fluctuations of the free-surface, which provides a more realistic and intricate fluid model in contrast to the `rigid-lid' approximation that is typically invoked in the analysis of two-layered fluids (see discussions in \cite{CI-19,CIM,CGK}).

In this paper, we show that the phase-plane, and particle trajectory, analysis of \cite{HV-JDE,HV-AN} can be extended to incorporate Coriolis effects modelling the Earth's rotation in the Equatorial region. The effect of Coriolis forces are appreciable for large scale flows, as shown in Section \ref{SecDisp} (cf. \cite{CJ15}). 
We restrict our analysis to the Equatorial region for a number of reasons.
 Stratification in the equatorial zone is more pronounced than in any other oceanic area \cite{FB, Gill}: an upper shallow layer of relatively warm water (and thus of lower density)  is separated from a deeper layer of colder, denser water by a very thin layer of the so-called pycnocline (or thermocline) where the density (and the temperature) gradient has a maximum. The density difference across the thermocline is $1\%$- $2\%$, and therefore it is reasonable to model the ocean stratification with two superimposed layers of constant density:  3–4 m high gravity waves are common on the surface of the ocean, while large internal waves (with heights in excess of 30 m) propagate as oscillations of the thermocline. 
Equatorial waves also exhibit particular dynamics due to the vanishing of the Coriolis parameter along the Equator, which effectively acts as a waveguide in the sense that disturbances are trapped in the vicinity of the Equator \cite{Gill, J18}. Field data shows that the meridional variations are very small, therefore it is reasonable to use a latitude-independent flow in the $f$-plane approximation with a vanishing meridional velocity.
This fact effectively allows to analyse the equatorial flow in a two dimensional plane, whose horizontal axis spans along the Equator.   
 
 The structure of the paper is as follows. In Section \ref{sec:2}, we introduce the governing equations for the geophysical equatorial flow of two fluid layers in the $f$-plane approximation.  In Section \ref{sec:3} we present the linearisation of the governing equations using a Hamiltonian formulation, and derive the linear wave solutions. In Section \ref{SecDisp} we derive, and analyse, dispersion relations for coupled surface and internal waves. In this section we prove a --- to our knowledge --- new result, Theorem \ref{THM}, which provides a  bound for solutions of the dispersion relations expressed in terms of  $A$ (the nondimensional parameter featuring the ratio of wave amplitudes), namely that $A\geq1$ must always hold. An important consequence of Theorem \ref{THM} is that it eliminates as a possibility one of the scenarios for motion of the upper fluid layer that was addressed in \cite{HV-JDE,HV-AN}.  In Section \ref{SecDyn} we present the nonlinear dynamical systems which describe the motion of the fluid particles in each layer, and in Section  \ref{secPP}  show that the phase-plane analysis of \cite{HV-JDE,HV-AN} applies to these. 
Finally, in Section \ref{TRAJ} a Lagrangian description of the fluid motion is provided, whereby the particle trajectories are elucidated.

\section{Governing equations }\label{sec:2}
This paper considers wave motion in the Equatorial region, which is assumed to be effectively two-dimensional due to Coriolis terms acting as a natural waveguide.
We work in Cartesian coordinates, with the positive $x-$coordinate representing the longitudinal direction, and the $y$ coordinate the local vertical direction.
The density  $\bm{\rho}$ of the fluid is assumed to be piecewise constant being distributed in the following way: we assume that there is an upper fluid layer of constant density $\rho_1$,
$$\Omega_1(\eta,\eta_1):=\{(x,y):x\in [0,L],\eta(x,t)<y<h_1+\eta_1(x,t)\},$$
which lies above a lower fluid layer  of constant density $\rho$
$$\Omega(\eta):=\{(x,y):x\in [0,L],-h<y<\eta(x,t)\},$$ 
where  $\rho\equiv (r+1)\rho_1>\rho_1$ for some $r>0$ (typical oceanographic values for $r$ in the Equatorial region are $\mathcal O(10^{-3}$).
Here $y=h_1+\eta_1(x,t)$ is the free-surface, and  $y=\eta(x,t)$  denotes the internal interface, while $y=-h$ is the location of the (impermeable) flat bed. We assume that waves are periodic, with period $L>0$: $\eta(x+L,t)=\eta(x,t)$, $\eta_1(x+L,t)=\eta_1(x,t)$. Further, taking
\[  \int_0^L \eta(x,t) dx = \int_0^L \eta_1(x,t) dx =0 \] for all $t$ ensures that the lower fluid layer has mean-depth $h>0$, with the upper-layer having mean-depth $h_1>0$.
With respect to the above described stratification the velocity field can be split as
\begin{equation*}
\bm{u}(x,y,t):=
\left\{\begin{array}{ccc}  u(x,y,t), &{\rm in} & \Omega(\eta),\\ u_1(x,y,t), &{\rm in} &\Omega_1(\eta,\eta_1),\end{array}\right.
\end{equation*}
and 
\begin{equation*}
\bm{v}(x,y,t):=
\left\{\begin{array}{ccc} v(x,y,t), &{\rm in} & \Omega(\eta),\\ v_1(x,y,t), &{\rm in} &\Omega_1(\eta,\eta_1).\end{array}\right.
\end{equation*}
\begin{figure}[H]
\begin{center}
 \resizebox{.7\textwidth}{!}{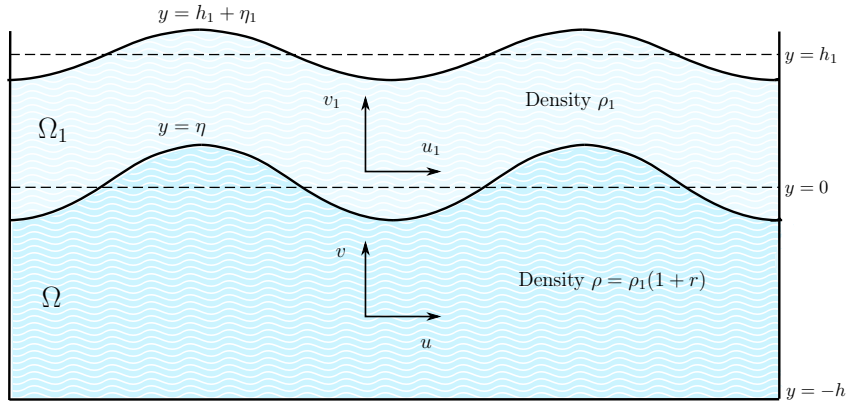}
\caption{Coupled ``in phase''  surface and internal water waves.}
\label{Fig1}
\end{center}
\end{figure}
Since inviscid flow is a natural assumption for ocean gravity waves, the equations of motion involve the Euler equation. We invoke the $f-$plane approximation for the full governing equations of geophysical fluid dynamics in a reference frame fixed at a point on the Earth's surface, and rotating with the Earth.
This approximation is applicable in the equatorial region when latitudinal variation about the equator is relatively small, and can be taken as fixed. The $f-$plane equations of motion are
\begin{equation*}
 \left\{\begin{array}{lcl}
        \bm{ u}_t+\bm{u}\bm{u}_x+\bm{v}\bm{u}_y+ f \bm{v} &=& -\bm{\frac{1}{\rho}}P_x,\\
         \bm{v}_t+\bm{u}\bm{v}_x+\bm{v}\bm{v}_y -f \bm{u} &=& -\bm{\frac{1}{\rho}}P_y-g,
        \end{array}\right.
\end{equation*}
where $P=P(x,y,t)$ denotes the pressure, $g$ is the gravitational acceleration, and $f=2\omega_E=1.5\cdot10^{-4}$ rad/s is the Coriolis parameter ($\omega_E=7.3\cdot10^{-5}$ rad/s is the Earth's angular velocity).
Typical values for oceanic flow speeds are $0.1$ m/s for the horizontal velocity components, and $10^{-5}$ m/s for the vertical velocity, and so Coriolis acceleration terms are about $10^{-6}$ m/s$^2$, with the centrifugal acceleration of (about $3 \times  10^{-2}$   m/s$^2$ ) several orders of magnitude larger, but still only 0.3\% of the gravitational acceleration. These numerical values indicate that the forces associated with the Earth’s rotation will only be relevant at large scales \cite{Constantin2021}. As we will see in Section \ref{SecDisp}, the impact of Coriolis forces on equatorial wave-motion is illustrated by the  fact that the difference between the magnitudes of the phase velocities for eastward and westward moving long-waves on the surface, for typical parameter values of the equatorial wave motion, is
\[
    | c_{+}^{(s)}|-| c_{-}^{(s)}   |= -0.66  \text{ m/s}.
\]
This difference is of the magnitude of the average speed of the Equatorial Undercurrent (EUC), which is of the order of 1 m/s.

The equation of incompressibility takes the form
\begin{equation}\label{masscons}
\bm{ u}_x+\bm{v}_y=0\,\,\textrm{in}\,\,\Omega\cup\Omega_1.
\end{equation}
Complementing the equations of motion are the free-boundary conditions. The dynamic boundary condition 
\begin{equation}\label{atm}
 P=P_{atm}\,\,{\textrm on}\,\, y=\eta_1 (x)+h_1,
\end{equation}
(with $P_{atm}$ the constant atmospheric pressure) decouples the motion of the water from that of the air. In addition to \eqref{atm} The kinematic boundary conditions
 reflect the impermeability of the flat bed $y=-h$, the interface $y=\eta(x,t)$, and the free surface $y=h_1+\eta_1(x,t)$, and are given by
\begin{equation}
 v=0 \quad{\rm on}\quad y=-h,
\end{equation}
 \begin{equation}
 \label{kin_int}
 \begin{array}{lll}
v_1=\eta_t +u_1 \eta_x & {\rm on} & y=\eta(x,t),\\
v=\eta_t+u\eta_x & {\rm on} & y=\eta(x,t),
\end{array}
\end{equation}
and 
 \begin{equation}
 \label{kin_fs}
 v_1=\eta_{1,t}+u_1\eta_{1,x}\,\,{\rm on}\,\,y=\eta_1(x,t)+h_1.
\end{equation}
Throughout the paper we assume that the flow in both fluid layers is irrotational, that is, the vorticity is zero in each separate layer:
 \[
 \bm{\gamma}:=\bm{u}_y-\bm{v}_x=0 \quad \mbox{ in } \Omega \cup \Omega_1.
 \] 
 This enables the introduction of a velocity potential for each fluid layer $$\bm{\varphi}(x,y,t)=\left\{\begin{array}{ccc} \varphi(x,y,t), &{\rm in}& \Omega,\\ \varphi_1(x,y,t), &{\rm in} & \Omega_1,\end{array}\right.$$
by means of
\begin{equation}\label{vel_pot}
 \left\{\begin{array}{lll}
  u=\varphi_x ,&v=\varphi_y, & {\rm in}\quad \Omega\\
  u_1=\varphi_{1,x}, &v_1=\varphi_{1,y},& {\rm in}\quad \Omega_1.
 \end{array}\right.
\end{equation} From \eqref{masscons} and \eqref{vel_pot} it follows that the functions $\varphi$ and $\varphi_1$ are harmonic in their domains,
 \begin{align}
   & \Delta \varphi(x,y,t)=\varphi_{xx}+\varphi_{yy}=0 \, \, {\rm in} \,\, \Omega, \nonumber \\ 
   &\Delta \varphi_1(x,y,t)= \varphi_{1,xx}+\varphi_{1,yy}=0     \, \, {\rm in} \, \, \Omega_1. \label{Laplace}
 \end{align} 
 Due to the lack of underlying currents and vorticities, it can be shown that the potentials $\varphi(x,t),$ $\varphi_1(x,t)$ are also $L$-periodic in $x.$ 
The kinematic boundary conditions  \eqref{kin_int} and \eqref{kin_fs} can now be recast in terms of velocity potentials as 
\begin{equation}\label{kin_interface}
 \begin{array}{c}
   \eta_t=( \varphi_{1,y})_{s}-\eta_x (\varphi_{1,x})_{s},\\
 \eta_t=(\varphi_y)_{s}- \eta_x (\varphi_x)_{s}   ,
  \end{array}
\end{equation}
and 
\begin{equation}\label{kin_surface}
 \eta_{1,t}=(\varphi_{1,y})_{s_1}-\eta_{1,x} (\varphi_{1,x})_{s_1}\end{equation}
where the subscript $s_1$ denotes the trace at the free surface, while the subscript $s$ denotes the trace at the interface.
The following notation will be used later:
\begin{equation}\label{Phi_notation}
\begin{array}{l}
 \Phi(x,t)=\varphi(x,\eta(x,t),t),\\
 \Phi_1(x,t)=\varphi_1(x,\eta(x,t),t),\\
 \Phi_2(x,t)=\varphi_1(x,h_1+\eta_1(x,t),t).\end{array}
\end{equation}

\section{Linearisation  via Hamiltonian formulation}\label{sec:3}

The systematic linearisation of the governing equations necessitates a suitable nondimensionalisation and scaling of the variables. In the situation where the depths $h$ and $h_1$ are of the same order of magnitude, so we can introduce a depth scale $ \mathfrak {h}$ such that $h,h_1= \mathcal O(\mathfrak{h})$; similarly for the  amplitudes of the surface and internal wave, $a$ and $a_1$, respectively --- we introduce the amplitude length scale $ \mathfrak {a}$ and assume that  $a,a_1= \mathcal O(\mathfrak{a})$. The standard scaling parameters are defined
\begin{equation}\label{param}
\epsilon= \mathfrak {a}/ \mathfrak {h}, \quad  \delta= k\mathfrak {h}, 
\end{equation}
for wavenumber $k=2\pi/L$, where $\epsilon$  measures the amplitude size relative to the layer depth, while $\delta$ measures the depth relative to the wavelength. In the linear regime it is necessary for $\epsilon $ to be very small ($\ll1$). At this stage we do not specify the wavelengths and assume that $\delta = \mathcal{O}(1)$.
These parameters can be used to nondimensionalise all quantities (cf. \cite{Johnson})
\begin{align}
\bar{ h} &=  \mathfrak{h} h ,  \quad  \bar{ h}_1 =  \mathfrak{h} h_1 ,  \quad \bar{x}= \frac{\mathfrak{h}}{\delta} x, \quad \bar{y}= \frac{\mathfrak{h}}{\delta^2} y ,\quad \bar{t} = \frac{\mathfrak{h} }{\delta \sqrt{\bar{g}\mathfrak{h}}}  t, \quad \bar{\eta} = \epsilon \mathfrak{h} \eta,  \nonumber \\
\bar{\varphi}&= \frac{\epsilon}{\delta } \mathfrak{h } \sqrt{\bar{g}\mathfrak{h}} \varphi, \quad 
\bar{\eta}_1 = \epsilon \mathfrak{h} \eta_1, \quad  \bar{\varphi}_1= \frac{\epsilon}{\delta } \mathfrak{h } \sqrt{\bar{g} \mathfrak{h}} \varphi_1, \quad \bar{f} =  \frac{ \sqrt{\bar{g}\mathfrak{h}} }{\mathfrak{h}}   f.
\end{align}
Bars are used to denote the dimensional quantities, while the nondimensional ones are without a bar. The nondimensional version of the Earth's acceleration $\bar{g}=9.81$ m/s$^2$ is  $g=1$, but we will retain $g$ throughout the following in order to easily revert back to the dimensional form of the governing equations. It is also assumed that the nondimensional constant $f=\mathcal{O}(1)$, although we note that since $ \bar{f} = 1.5\times 10^{-4}$ rad/s, even for significant depths $f$ will be quite small: for instance, if $\mathfrak{h}=4\times 10^4$m, then $f=0.01$.

The equations \eqref{kin_surface}--\eqref{Phi_notation} become (for the non-dimensional variables)
\begin{align*}
 &\eta_{1,t}=(\varphi_{1,y})_{s_1}-  \epsilon \eta_{1,x}  (\varphi_{1,x})_{s_1}  
 \\
\nonumber    &\eta_t=( \varphi_{1,y})_{s}-  \epsilon \eta_x  (\varphi_{1,x})_{s}, \ \eta_t=(\varphi_y)_{s}- \epsilon \eta_x (\varphi_x)_{s},
\\
\nonumber &\Phi(x,t)=\varphi(x,\epsilon \eta(x,t),t), \ \Phi_1(x,t)=\varphi_1(x, \epsilon \eta(x,t),t),  \ \Phi_2(x,t)=\varphi_1(x,h_1+ \epsilon \eta_1(x,t),t).
 \end{align*} 
In the leading order, as $\epsilon \to 0$, the above equations are approximated  by
\begin{equation} \label{kin_surface1}
 \eta_{1,t}=(\varphi_{1,y})_{s_1}  ; 
 \end{equation}
\begin{equation}\label{kin_interface1}
   \eta_t=( \varphi_{1,y})_{s}, \quad 
 \eta_t=(\varphi_y)_{s},
\end{equation}
and
\begin{align}\nonumber     
 \Phi(x,t)=\varphi(x,0,t),\quad
 \Phi_1(x,t)=\varphi_1(x,0,t),\quad
 \Phi_2(x,t)=\varphi_1(x,h_1,t).\end{align}
The governing equations can be represented in a Hamiltonian form along the lines of \cite{CI-19,CIM}, which takes into account both the Coriolis force, and constant vorticities in each layer. Here we are include the Coriolis force (without vorticity), and present the salient details for the sake of completeness.  The Hamiltonian variables are $\eta, \eta_1$ and
\begin{equation} \label{HV1}
\begin{split}
&\xi=\rho \varphi(x,0,t) -\rho_1 \varphi_1(x,0,t)\equiv \rho_1\big((r+1)\Phi(x,t)-\Phi_1(x,t) \big) , \\
& \xi_1=\rho_1\varphi_1(x,h_1,t)\equiv\rho_1 \Phi_2(x,t);
\end{split}
\end{equation} where the non-dimensional quantity $r$ is defined via $\rho = \rho_1(1+r).$ The governing equations have the form   \cite{CI-19}
\begin{equation} \nonumber
\begin{split}
\xi_t=&-\frac{\delta H}{\delta \eta}-\Gamma\int_{-\infty}^{x}\frac{\delta H}{\delta \xi(x')}dx' \\
\eta_t=&\frac{\delta H}{\delta \xi}\\
\xi_{1,t}=&-\frac{\delta H}{\delta \eta_1}-\Gamma_1\int_{-\infty}^{x}\frac{\delta H}{\delta \xi_1(x')}dx' \\
\eta_{1,t}=&\frac{\delta H}{\delta \xi_1}
\end{split}
\end{equation}
where $\Gamma= (\rho -\rho_1)f,$ $  \Gamma_1=\rho_1 f$ are constants and $H[\xi,\xi_1, \eta, \eta_1]$ is the Hamiltonian functional. The Hamiltonian is given by the energy of the system for one period (horizontal length $L$), and has the same integral density as in  \cite{CIM,CI-19}, given by the expression 
\begin{equation}
\begin{split}
H[\xi,\xi_1, \eta, \eta_1]=&\frac{1}{2}\int_{0}^L\Big[\xi \frac{D\tanh(hD)\coth(h_1D)}{\rho \coth(h_1D)+\rho_1 \tanh(hD)}\xi +2\xi \frac{D\tanh(hD)\text{csch}(h_1D)}{\rho \coth(h_1D)+\rho_1 \tanh(hD)}\xi_1\\
+&\xi_1 \frac{D\big(\tanh(hD)\coth(h_1D)+\frac{\rho}{\rho_1}\big)}{\rho \coth(h_1D)+\rho_1 \tanh(hD)}\xi_1 +g(\rho-\rho_1)\eta^2+\rho_1 g\eta_1^2  \Big]dx .
\end{split}
\end{equation}
The explicit form of the equations is therefore 
\begin{equation} \label{linsys11}
\begin{split}
\xi_t=&-g(\rho-\rho_1)\eta-\Gamma \partial^{-1}\eta_t, \\
\eta_t=& \frac{D\tanh(hD)\coth(h_1D)}{\rho \coth(h_1D)+\rho_1 \tanh(hD)}\xi + \frac{D\tanh(hD)\text{csch}(h_1D)}{\rho \coth(h_1D)+\rho_1 \tanh(hD)}\xi_1,\\
\xi_{1,t}=&-\rho_1g\eta_1-\Gamma_1 \partial^{-1}\eta_{1,t}, \\
\eta_{1,t}=&\frac{D\tanh(hD)\text{csch}(h_1D)}{\rho \coth(h_1D)+\rho_1 \tanh(hD)}\xi+ \frac{D\big(\tanh(hD)\coth(h_1D)+\frac{\rho}{\rho_1}\big)}{\rho \coth(h_1D)+\rho_1 \tanh(hD)}\xi_1.
\end{split}
\end{equation}
These equations will be used to solve the boundary value problems for the velocity potentials in the two layers, and also to obtain the dispersion relation for our system.
To explicitly determine the velocity potentials in each fluid layer from  their values on the surface and on the interface corresponds to analytic continuation procedure, using the fact that the potentials satisfy the Laplace equation \eqref{Laplace} inside their fluid domains. The boundary values will be obtained from \eqref{kin_surface1} and \eqref{kin_interface1} by adopting the standard linear Ansatz for the functions $\eta(x,t)$ and $\eta_1(x,t)$. The corresponding boundary value problem for the lower layer therefore is
\begin{equation} \label{LL1}
\begin{split}
&\Delta \varphi(x,y) =0, \\
&\frac{\partial \varphi}{\partial \mathbf{n}} =0 \,\, \text{on} \,\, y=-h,\\
& \varphi_y(x,0,t)=\eta_t(x,t).
 \end{split}
\end{equation}
We consider the periodic problem (as opposed to \cite{CI-19,CIM}, which are focused on the real line) and so the operator $D=-i\partial_x$ now has only discrete spectrum, $\tilde{k}=\tilde{k}_n=\frac{2\pi n}{L}$ for an integer $n,$ and periodic eigenfunctions, $e^{i\tilde{k}x},$ since
\[  D e^{i \tilde{k}x} = \tilde{k} e^{i\tilde{k}x}. \]
For convenience,  summation over the integers $n$ (from $-\infty$ to $\infty$) will be written simply as a summation over $\tilde{k},$ remembering the correspondence between the two quantities.
The $t$-dependence is not written explicitly, we can treat $t$ here as a parameter, while the problem is stated in the $(x,y)$ domain. 
The general solution in the bulk is a superposition of harmonic functions $\varphi_{\tilde{k}}(x,y),$  periodic in $x,$ satisfying \eqref{LL1}
\begin{equation} \nonumber 
    \varphi(x,y)=\sum_{\tilde{k}} \mathcal{A}(\tilde{k}) \varphi_{\tilde{k}}(x,y) , \quad \varphi_{\tilde{k}}(x,y)=e^{i \tilde{k}x}\cosh [\tilde{k}(y+h)].
\end{equation}
The yet unknown function $\mathcal{A}(\tilde{k})$ could be obtained from the interface condition
\begin{equation} \nonumber 
\eta_t(x)=\varphi_y(x,0)=\sum_{\tilde{k}}  \tilde{k}\mathcal{A}(\tilde{k}) e^{i\tilde{k}x}\sinh (\tilde{k} h) .
\end{equation}
The Fourier coefficient of the above expansion is given by
\begin{equation} \nonumber 
\begin{split}
&\mathcal{A}(\tilde{k})=\frac{1}{L \tilde{k}\sinh (\tilde{k} h) }\int_0^L e^{-i\tilde{k}x'}\eta_t(x')dx',\\
&\varphi(x,y)=\frac{1}{L}\sum_{\tilde{k}}\int_0 ^L e^{i\tilde{k}(x-x')} \eta_t(x') \frac{\cosh[\tilde{k}(y+h)]}{\tilde{k}\sinh(\tilde{k}h)} dx'.
\end{split}
\end{equation}
The periodic delta-function 
$$\delta(x-x')=\frac{1}{L}\sum_{\tilde{k}} e^{i\tilde{k}(x-x')}  $$ 
is normalised by the condition $\int_0^L \delta(x) dx=1$ and we have
\begin{equation} \nonumber
\begin{split}
&\varphi(x,y)= \frac{\cosh[D(y+h)]}{D\sinh(Dh)} \int_0 ^L \frac{1}{L} \sum_{\tilde{k}} e^{i\tilde{k}(x-x')} \eta_t(x') dx' = \frac{\cosh[D(y+h)]}{D\sinh(Dh)} \eta_t(x).
\end{split}
\end{equation}
For the particular choice $\eta = a\cos(kx-\omega t)$, where $k= 2\pi /L$ is the wavenumber, we note that \[D^2 \sin(kx -\omega t) = k^2 \sin(k x - \omega t)\] and similarly for cosine. Thus, for even powers, the action of $D$ can formally be replaced by $k$, that is
\begin{equation}\label{phisol}
    \varphi(x,y,t)= a c(k) \frac{\cosh[k(y+h)]}{\sinh(kh)} \sin(kx-\omega t),
\end{equation}
and also 
\begin{equation} \label{Phi}
    \Phi(x,t):=\varphi(x,0,t)=ac \coth(kh) \sin(kx-\omega t).
\end{equation}
The dispersion relation $c(k)=\omega(k)/k$ (giving the possible propagation speeds) will be determined in Section \ref{SecDisp}.
The boundary-value problem for the upper layer is
\begin{equation} \nonumber 
\begin{split}
&\Delta \varphi_1(x,y) =0, \\
&\varphi_{1,y} (x, h_1)=\eta_{1,t}(x) \,\, \text{on} \,\, y=h_1,\\
& \varphi_{1,y}(x,0)=\eta_t(x).
 \end{split}
\end{equation}
The general solution in the bulk is a superposition of harmonic functions of the form $e^{\pm \tilde{k}y+i\tilde{k}x},$ periodic in $x,$ that is
\begin{equation} \nonumber 
\varphi_1(x,y)=\sum_{\tilde{k}} \left(\mathcal{A}(\tilde{k})e^{\tilde{k}y+i\tilde{k}x}+\mathcal{B}(\tilde{k})e^{-\tilde{k}y+i\tilde{k}x} \right) 
\end{equation}
The yet unknown Fourier coefficients $\mathcal{A}(\tilde{k})$ and $\mathcal{B}(\tilde{k})$ can be obtained from the surface and interface conditions.
\begin{equation} \nonumber 
\begin{split}
&\eta_{1,t}(x)=\varphi_{1,y}(x,h_1)=\sum_{\tilde{k}} e^{i\tilde{k}x} \tilde{k} \left( \mathcal{A}(\tilde{k})e^{\tilde{k}h_1} -\mathcal{B}(\tilde{k})e^{-\tilde{k}h_1} \right),\\
&\eta_t(x)=\varphi_{1,y}(x,0)=\sum_{\tilde{k}} e^{i\tilde{k}x}  \tilde{k} \left( \mathcal{A}(\tilde{k}) -\mathcal{B}(\tilde{k})\right) ,\\
\end{split}
\end{equation}
Evaluating the coefficients of the Fourier series given above allow us to determine $\mathcal{A}(\tilde{k})$ and $\mathcal{B}(\tilde{k}):$
\begin{equation} \nonumber
\mathcal{A}(\tilde{k})= \frac{1}{L}\int_{0}^L \!\!\frac{\eta_{1,t}(x')-e^{-\tilde{k}h_1}\eta_t(x')}{2\tilde{k}\sinh(\tilde{k}h_1)}e^{-i\tilde{k}x'} dx',\quad 
 \mathcal{B}(\tilde{k})= \frac{1}{L}\int_0^L \!\! \frac{\eta_{1,t}(x')-e^{\tilde{k}h_1}\eta_t(x')}{2\tilde{k}\sinh(\tilde{k}h_1)} e^{-i\tilde{k}x'}dx',
\end{equation}
and proceeding like in the solution of the previous problem we obtain finally
\begin{equation} \nonumber 
\varphi_1(x,y)=\frac{\cosh(yD)}{D\sinh (h_1D)}\eta_{1,t}(x)-\frac{\cosh((y-h_1)D)}{D\sinh( h_1D)}\eta_t(x).
\end{equation}
In particular, assuming the standard Ansatz  for linear water waves
\begin{equation}\label{eta-ans}
\eta=a \cos(kx-\omega t) \quad \mbox{ and } \quad \eta_1=a _1 \cos(kx-\omega t),
\end{equation}
gives
\begin{equation} \label{UL51}
\varphi_1(x,y,t)=\frac{c}{\sinh (k h_1)}[ a_1 \cosh(ky)- a\cosh(k(y-h_1))]\sin(kx-\omega t).
\end{equation}
On the boundaries of the upper domain we have the following 
\begin{equation}\label{Phi1}
\begin{split}
    &\Phi_1(x,t):= \varphi_1(x,0,t)=\frac{c}{\sinh (k h_1)}[ a_1 - a\cosh(kh_1)]\sin(kx-\omega t),\\
    &\Phi_2(x,t):= \varphi_1(x,h_1,t)=\frac{c}{\sinh (k h_1)}[ a_1 \cosh(kh_1)- a]\sin(kx-\omega t).
\end{split}
    \end{equation}
From the linear water wave Ansatz \eqref{eta-ans}, and the first and third equations of \eqref{linsys11}, we obtain 
\begin{equation} \label{HV3}
\begin{split}
& \xi =\frac{a}{\omega}\left( g(\rho-\rho_1)-\frac{\Gamma \omega}{k}  \right) \sin(kx-\omega t)=\frac{a\rho_1}{c k} r \left( g- f c  \right) \sin(kx-\omega t),\\
& \xi_1 =\frac{a_1}{\omega}\left( g\rho_1-\frac{\Gamma_1 \omega}{k}  \right) \sin(kx-\omega t)=\frac{a_1\rho_1}{c k}\left( g- f c  \right) \sin(kx-\omega t),
\end{split}
\end{equation}
where we recall that the Hamiltonian variables are as in \eqref{HV1}.
Introducing the new variable \begin{equation} \label{A}
    A:=\frac{g-f c }{kc^2},
\end{equation}
the substitution of \eqref{HV3}, \eqref{Phi1} and \eqref{Phi} in \eqref{HV1} gives
\begin{equation} \label{RR}
\begin{split}
& [(r+1) \coth(kh)+\coth(kh_1)-rA]a =\frac{a_1}{\sinh (kh_1)} ,\\
& [\coth(kh_1) - A]a_1 =\frac{a}{\ \sinh(kh_1)}.
\end{split}
\end{equation}
The second equation  leads to the expressions for the amplitude ratio $a/a_1$ (and $A$):
\begin{equation}\label{A-a}
    \frac{a}{a_1}=\cosh(kh_1)-A\sinh(kh_1) \Longleftrightarrow
A=\left(1-\frac{a}{a_1}{\sech(kh_1)}\right)\coth(kh_1).
\end{equation}
Taking into account \eqref{A-a} we can write \eqref{UL51} as 
\begin{equation} \label{phi1-A}
\varphi_1(x,y,t)=a_1 c[ \sinh(k(y-h_1))+A\cosh(k(y-h_1))]\sin(kx-\omega t).
\end{equation}
Thus, we have obtained the potentials in \eqref{phisol} and \eqref{phi1-A}. The dispersion relation expression $c=c(k)$ will be derived in the next section. 

\section{Dispersion relations} \label{SecDisp}
\subsection{Dispersion relation for amplitude ratio parameters \texorpdfstring{$A$}{A} }
Ensuring compatibility between both equations in \eqref{RR} leads to the quadratic equation
\begin{equation}\label{eqA}
    rA^2-(r+1)[\coth(kh)+\coth(kh_1)]A+(r+1)\coth(kh)\coth(kh_1)+1=0.
\end{equation}
This equation represents a dispersion relation that must be satisfied by $A$ (as defined by \eqref{A-a}), which in turn prescribes the ratio of the amplitudes of the wave solutions given by \eqref{eta-ans}. The dispersion relation \eqref{eqA} has the two real, and positive, roots
\begin{equation}
\label{solutionA}
    A_{\pm}=\frac{2(r + 1)\sinh(k(h + h_1)) \pm \sqrt{\Delta}}{4 r\sinh(k h)\sinh(k h_1)}, 
    \end{equation}
\begin{align*}
    \Delta:=& 8 r\cosh(k(h + h_1))\cosh(k(h - h_1)) + 2r^2 \cosh(2k(h - h_1)) + 2\cosh(2k(h + h_1))\\
    & \qquad \qquad \qquad \qquad  \qquad \qquad \qquad \qquad \qquad \qquad \qquad \qquad \qquad\quad -( 2r^2 + 8r + 2).
\end{align*} 
Denoting  \eqref{eqA} as $\mathcal P(A)=0$, it follows by direct calculation that
\[
\mathcal P\left(\frac{1}{\tanh kh_1}\right)=1-\frac{1}{\tanh^2 kh_1}<0, \qquad    \mathcal P\left(\frac{1}{\tanh kh}\right)=1-\frac{1}{\tanh^2 kh_1}<0,
\]
giving the relations (which we note are independent of the relative densities):
\begin{equation}\label{tauineq}
A_-<\frac{1}{\tanh kh_1} <A_+, \qquad A_-<\frac{1}{\tanh kh} <A_+.
\end{equation}
It follows from \eqref{A} that $A=A_+$ corresponds to the slower propagating waves, while the solution $A=A_-$ corresponds to the faster propagating waves. Furthermore,   \eqref{A-a} and \eqref{tauineq} imply that  ${a}/{a_1}  >0 $ (waves are in-phase) for the faster propagating waves $A=A_-$, while ${a}/{a_1}  <0 $ (waves are out-of-phase) for the slower propagating waves $A=A_+$.

It is clear from \eqref{tauineq} that $A_+>1$ always, and we infer from \eqref{A} that $A_->0$ for all realistic wavespeeds (for typical parameter values of equatorial wave motion the ratio $fc/g$   is  less that $1\%$). However, since $A$ is not defined in isolation by \eqref{A} but must also satisfy the compatibility condition \eqref{eqA}, it turns out that a stronger condition holds.
\begin{theorem}\label{THM}
The roots $A_{\pm}$ of \eqref{eqA} are such that $A_{\pm}\geq 1$, with equality holding only for the root $A_{-}(k,h,h_1)$ in the limiting case whereby either $kh$, or $kh_1$, tends to infinity.
\end{theorem}
\begin{proof} 
For convenience, and without loss of generality, in this proof we assume that $h,h_1$ are fixed, and finite, and examine the limiting behaviour of $A_{\pm}$ with respect to $k$. Note that the limit of either $h$, or $h_1$, tending to infinity is invoked in a physical regime whereby either the lower, or upper, fluid layer is significantly larger than the other, respectively.
 
We observe from \eqref{solutionA} that the roots $A_{\pm}(k)$ of \eqref{eqA} are odd functions with respect to $k$, and $\lim_{k \to +0}A_{\pm}(k) =\infty.$ As $k \to \infty$  it follows directly from \eqref{eqA} that 
\begin{equation} \label{limA}
    \lim_{k \to \infty}A_+(k)=1+\frac{2}{r}>1, \quad \lim_{k \to \infty}A_-(k)=1.
\end{equation} 
By establishing that $A'_{\pm}(k) \ne 0$ for positive $k$, it will follows that the roots $A_{\pm}(k)$ decay monotonically to the limits in \eqref{limA}. We use the notation $A(k)$ to denote either $A_+(k)$ or $A_{-}(k)$. Let us assume the opposite, namely that $A'(k^*) = 0$ for some $k^*>0$.  Expressing \eqref{eqA} in the form 
\begin{equation}
      rA^2+\beta A+\gamma=0, \label{eqAshort}
     \end{equation}
where  $\beta(k)=-(r+1)[\coth(kh)+\coth(kh_1)]$, and $\gamma(k)=(r+1)\coth(kh)\coth(kh_1)+1$, then differentiating \eqref{eqAshort} with respect to $k$, and using $A'(k^*) = 0$, gives
\begin{equation} \label{Anew}
    A(k^*)= - \frac{\gamma'(k^*)}{\beta'(k^*)}.
\end{equation}
Equations \eqref{eqAshort} and \eqref{Anew} lead to the following relation which holds at the critical point:
\begin{equation} \nonumber
   F(k^*)\equiv  r [\gamma'(k^*)]^2 - \beta (k^*) \beta' (k^*)\gamma'(k^*) + \gamma (k^*) [\beta'(k^*)]^2=0,
\end{equation}
and direct computation shows that 
\begin{equation} \nonumber
    F(k)=F_1(k) (r + 1)^2 \text{csch}^4(k h_1) \text{csch}^4(k h),
\end{equation}
where
\begin{align} \nonumber
    F_1(k)&=F_{10}(k)+ r h h_1 F_{11}(k), \\ \nonumber
    F_{10}(k)& =h h_1 [\cosh^2(k h_1) + \cosh^2(k h)] + h^2 \sinh^2(kh_1)  + h_1 ^2\sinh^2(k h)  - 2 hh_1, \\  \nonumber
    F_{11}(k)& =2 \cosh(k h_1) \cosh(k h) \cosh(k(h - h_1)) - [\cosh^2(k h_1) +\cosh^2(k h)].
\end{align}
   From the Cauchy-Schwarz inequality
   \begin{align}
       F_{10}(k)& \ge 2 h h_1 \cosh(k h_1)  \cosh (k h) + 2 h h_1 \sinh(kh_1)  \sinh(k h)  - 2 hh_1 \nonumber \\
       &= 2 h h_1 [\cosh(k ( h+h_1) ) -1] > 0.     \nonumber
   \end{align}
Let us now analyse $F_{11}(k).$ We observe that $F_{11}(k)=0$ in the particular case $h=h_1.$ In this case nevertheless $F_{10}(k)>0$ and hence $F(k)>0.$ 
If $h_1 \ne h$, then  $F_{11}(0)=0$ and $$F'_{11}(k)=(h-h_1) \sinh(2k(h-h_1))>0$$ for $k>0.$
Therefore $F_{11}(k)>0$ for $k>0$ and again $F(k)>0.$ This contradiction shows that the equation $F(k^*)=0$ does not have any solutions for  $k^*>0$, and hence $A(k)$ has no turning points: $A(k)$ must decrease monotonically from $\infty$ to the limiting values \eqref{limA}.
\end{proof}
\begin{corollary}
 It follows by Theorem \ref{THM} and  \eqref{A} that wavespeeds $c$ are bounded by
\[
c^2+\frac{f}{k}c<\frac{g}{k}.
\]
In the absence of Coriolis forces ($f=0$), this inequality reduces to $c^2<g/k$.
\end{corollary}
\subsection{Dispersion relation for wavespeeds \texorpdfstring{$c$}{c} }
Combining \eqref{A} with \eqref{eqA} leads to a corresponding dispersion relation for the wavespeeds $c(k)$ which has the form of a $4^{th}$ order polynomial.
In the absence of Coriolis forces ($f=0$) this quartic reduces to the well known dispersion relation (cf. \cite{Suth})
\begin{equation}     \nonumber  
c^4[(r+1) \coth(hk) \coth(h_1k)+1]-(r+1)\frac{g}{k}[\coth(hk) +\coth(h_1k)]c^2+r\frac{g^2}{k^2}=0.
\end{equation}
Alternatively, given the roots $A_{\pm}$ in \eqref{solutionA}, we can solve the quadratic equation \eqref{A} for $c$  to explicitly obtain all 4 possible linear wave mode solutions. Defining  the (even) function $B(k):=kA(k)$, we re-express  \eqref{A} in the form of a quadratic equation for $c$ with coefficients that are even functions of $k:$ 
\begin{equation}
     B(k) c^2 +fc - g = 0.
\end{equation}
The solutions to this equation 
\begin{equation}\label{ck-expl}
     c(k)=\frac{-f\pm \sqrt{f^2+4gB(k)}}{2B(k)}
\end{equation}
are real, since the discriminant $f^2+4gB(k)>0$ due to  $B(k)\geq 0.$ One can show that $B(k)\geq 0$ as follows. Since  $\sinh(k h)\sinh(k h_1)\geq 0,$ one needs to check the inequality $4(r+1)^2 \sinh^2 (k(h + h_1))  \geq \Delta(k^2),$ or
\begin{align}
        4(r+1)^2 [\cosh^2 (k(h + h_1))-1] &\geq 8 r[\cosh(k(h + h_1))\cosh(k(h - h_1))-1] \nonumber \\
        &+ 4r^2 [\cosh^2(k(h - h_1)) -1] + 4[\cosh^2(k(h + h_1))-1]  .   \nonumber
\end{align}
This can be further transformed to the form
\begin{align}
        (r^2+2r)[\cosh^2 (k(h + h_1))-1] &\geq 2 r[\cosh(k(h + h_1))\cosh(k(h - h_1))-1] \nonumber \\
        &+ r^2 [\cosh^2(k(h - h_1)) -1]   \nonumber
\end{align}
in which the inequality is evident since $\cosh(\mathcal{H}_1k) \geq \cosh(\mathcal{H}_2k) $ when $|\mathcal{H}_1|>|\mathcal{H}_2| .$
Therefore, all 4 solutions of \eqref{ck-expl} are real and even functions of $k.$
Moreover, there are always two positive solutions, representing right-running waves (in our situation, eastward travelling waves) $c_+^{(s)}(k), c_+^{(i)}(k)$ and two negative solutions for the left-running (westward) waves, $c_-^{(s)}(k), c_-^{(i)}(k)$. The  superscripts $s$ and $i$ denote the ``surface'' and ``internal'' propagation modes (also known as barotropic, and baroclinic, modes, respectively) and are related to the two possible solutions for $B(k)$.
In general, the surface modes have much larger speeds than the internal modes, $|c_{\pm}^{(s)}(k)|\gg | c_{\pm}^{(i)}(k)|.$

The short-wave asymptotics (the limit $k\to \infty$) are as follows:
\begin{equation}  \nonumber
  c_{\pm}^{(s)}(k)\simeq \pm \sqrt{\frac{g}{|k|}},\quad c_{\pm}^{(i)}(k)\simeq \pm \sqrt{\frac{rg}{(r+2)|k|}}.
\end{equation}
It is evident that the short wave limit does not depend on the Coriolis forces.
The long-wave limits ($k\to 0$) are constant values, which do depend on the Coriolis parameter $f:$
\begin{align}  \nonumber
     c_{\pm}^{(i)}&= \frac{-fr h h_1\pm \sqrt{f^2r^2 h^2 h_1^2+2 r g h h_1 [(r+1)(h+h_1)+\sqrt{\Delta_0}]}}{(r+1)(h+h_1)+\sqrt{\Delta_0}} ,\\  \nonumber
      c_{\pm}^{(s)}&=\frac{-fr h h_1\pm \sqrt{f^2r^2 h^2 h_1^2+2 r g h h_1 [(r+1)(h+h_1)-\sqrt{\Delta_0}]}}{(r+1)(h+h_1)-\sqrt{\Delta_0}} , \\
      \Delta_0:&=(r+1)^2(h^2+ h_1^2)+2(1-r^2)h h_1  .  \nonumber
\end{align}
For specific typical realistic values of the physical parameters, for example, $r = 0.005,$ 
$g = 9.81$ m/s$^2$, $h = 4000.0$ m, $ h_1 = 400.0$ m,  $f=2\omega_E=1.5\cdot10^{-4}$ rad/s, we have the following values (in m/s),
$c_{+}^{(s)}=    207.39$, $c_{+}^{(i)}=  4.21352$, $c_{-}^{(s)}= -208.05$, $c_{-}^{(i)}=-4.21379$.
The magnitude differences are
\begin{equation} \nonumber
    | c_{+}^{(s)}|-| c_{-}^{(s)}   |= -0.66  \text{ m/s}, \quad   | c_{+}^{(i)}|-| c_{-}^{(i)}   |= -2.7\times 10^{-4}\text{ m/s}.
\end{equation}
The difference for the internal modes is essentially unmeasurable, however for the surface modes the difference is of the magnitude of the average speed of the Equatorial Undercurrent (EUC), which is of the order of 1 m/s. In both cases, the westward component is larger. 
The complete solution of the 4 speeds as a function of $k$ is illustrated in Fig. \ref{fig:dispersion} where the parameter values are chosen with a larger $f$ to illustrate the features of the dispersion: $r = 0.01,$ $f = 1.0\times 10^{-2}$ rad/s,  $g = 9.81$ m/s$^2,$ $h = 500.0$ m   $h_1 := 100.0$ m. It is evident that $| c_{-}^{(s)}(k)|>| c_{+}^{(s)} (k)  |.$ 
In the first panel the results are in logarithmic scale for $k,$ in the second $k$ is in the usual scale and the behaviour near the origin is highlighted. 

We observe that the magnitude of the wave-speed of the surface waves drops significantly for $k>0.5,$ while the corresponding  magnitude  for the internal waves drops significantly for $k>0.05$.  This reflects the fact that the internal waves with large enough  magnitude and energy are observed for large wavelengths --- from hundreds of metres to several kilometres.
\begin{figure}[H]
	\begin{center}
		($a$)\ \fbox{\includegraphics[totalheight=0.25\textheight]{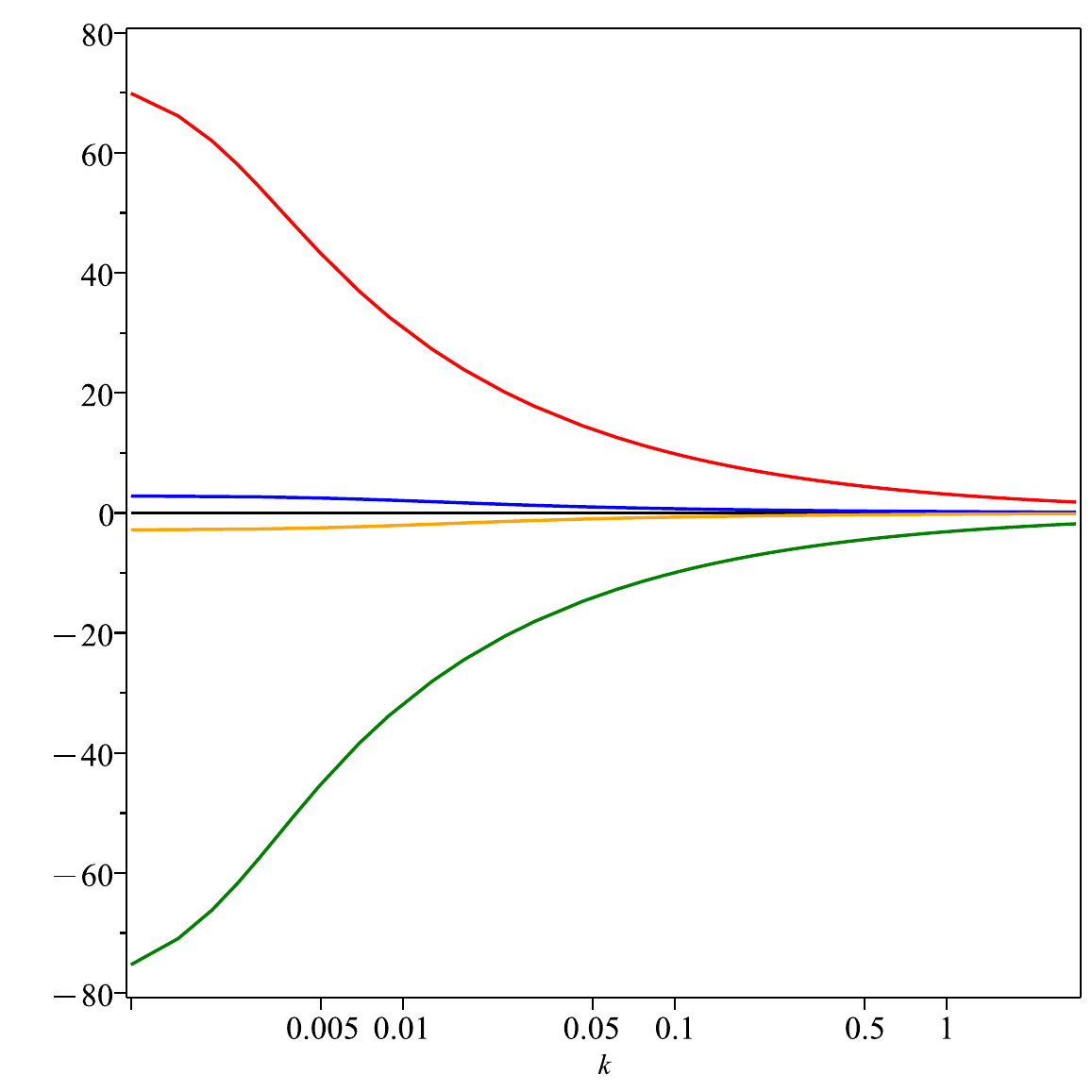}}\qquad 
		($b$)\ \fbox{\includegraphics[totalheight=0.25\textheight]{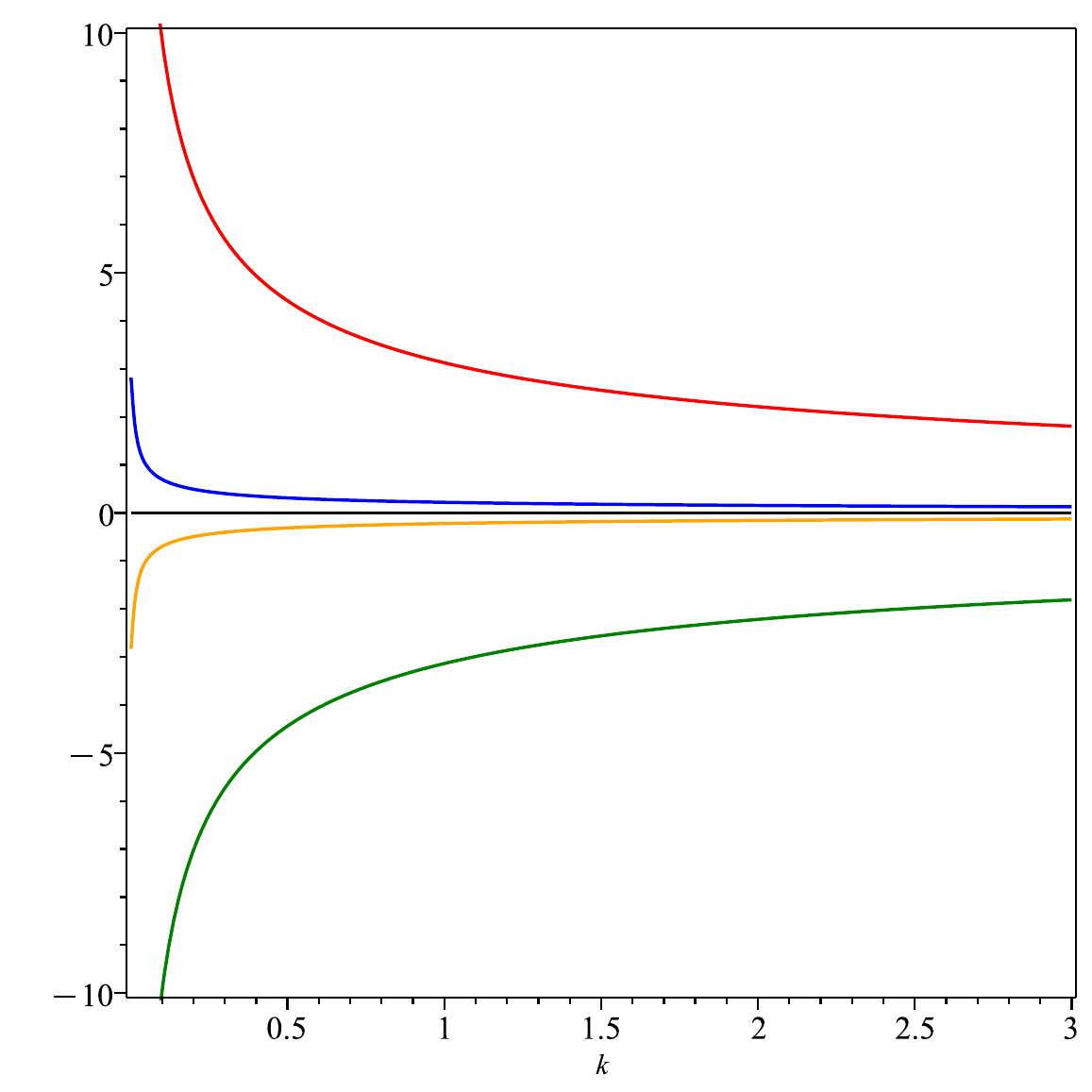}}
				\caption{The dispersion relations  $c_{+}^{(s)}(k)$ (red) $c_{-}^{(s)}(k)$ (green) 
                 $c_{+}^{(i)}(k)$ (blue) $c_{-}^{(i)}(k)$ (yellow) for the following parameter choice:
                $r = 0.01$, $f = 1.0\times 10^{-2}$ rad/s,  $g = 9.81$ m/s$^2,$ $h = 500.0$ m   $h_1 := 100.0$ m. In panel (a) $k$ is in logarithmic scale.  In panel (b) the graph is magnified near the origin of the vertical axis,  $k$ is in non-logarithmic scale.       }
		\label{fig:dispersion}
	\end{center}
\end{figure}

\section{Dynamical systems formulation}\label{SecDyn}
From the definitions \eqref{vel_pot}, and the solutions \eqref{phisol}, \eqref{phi1-A} we can directly compute expressions for velocity fields in the lower, and upper, fluid layers.  If $(x(t),y(t))$ denotes the trajectory of a fluid particle in the lower-fluid layer $\Omega$, with initial data $(x_0,y_0)$, then
\begin{eqnarray}
 \left\{\begin{array}{l}
   \frac{dx}{dt}=u=a\omega\cos (kx-\omega t)\frac{\cosh k(y+h)}{\sinh kh}  \\
   \phantom{5}
   \\ \label{LowSys1}
 \frac{dy}{dt}=v=a\omega \sin (kx-\omega t)\frac{\sinh k(y+h)}{\sinh kh},
  \end{array}\right.
\end{eqnarray}
for $-h<y<0$, while in the upper-fluid layer $\Omega_1$
\begin{equation}
 \left\{\begin{array}{lll}
  \frac{dx}{dt}=u_{1}=a_1\omega\cos (kx-\omega t)\left\{\sinh k(y-h_1)+A\cosh k(y-h_1) \right\} \\ \label{UpSys1}
 \frac{dy}{dt}=v_1=a_1\omega\sin (kx-\omega t)\left\{\cosh k(y-h_1)+A\sinh k(y-h_1) \right\},
  \end{array}\right.
\end{equation}
for $0<y<h_1$. The mean-level of the oscillating internal wave interface $y=\eta$ is located at $y=0$, whereas the free-surface $y=h_1+\eta_1$ oscillates about the mean-level located at $y=h_1$. 
Both \eqref{LowSys1} and \eqref{UpSys1} are nonlinear   and nonautonomous dynamical systems.

 Since the fluid layers are separated by an impermeable interface $y=\eta(x,t)$, and the solutions \eqref{phisol} and \eqref{phi1-A} satisfy matching conditions at this interface by design, we can address the phase plane analysis of system \eqref{LowSys1}  in the lower-fluid layer $\Omega$, and system \eqref{UpSys1} in the upper-fluid layer $\Omega_1$, separately in the first instance, and then piece together the information to get a picture of the motion of the entire two-layer body.

\section{Phase plane analysis}\label{secPP}
In order to  perform phase-plane analyses, the wavenumber $k$ and frequency $\omega$ are treated as fixed constants in the dynamical systems \eqref{LowSys1} and \eqref{UpSys1}, while  we vary the nondimensional parameter $A$ in  \eqref{UpSys1}, as required. It was shown in \cite{HV-JDE,HV-AN} that the dynamical system  \eqref{UpSys1} describing motion in the upper fluid layer has qualitatively different configurations, depending on whether $A<1$, $A=1$, or $A>1$. Theorem \ref{THM} now enables us to rule out the possibility that $A<1$ since it always holds that $A\geq1$.
 Of course, as detailed in Section \ref{SecDisp}, for a given wave motion $k$, $\omega$ and $A$ are neither fixed nor  free parameters  but, rather, must be determined by solving dispersion relations \eqref{A} and \eqref{eqA}.  

\subsection{Lower-fluid layer}\label{secPPlow}
Since the lower layer possesses just one moving boundary (the internal wave) compared to the upper layer's  two (the surface and internal waves), it is unsurprising that  the velocity field for the lower-fluid layer \eqref{LowSys1} is less complex in its prescription than that of the upper layer \eqref{UpSys1}. From a mathematical viewpoint, system \eqref{LowSys1} is qualitatively identical to that which describes fluid motion in a single homogenous  (uniform density) fluid layer whose upper interface separates the fluid from a source of constant pressure (such as  the atmosphere), cf. \cite{CV}.   The physical influence of the upper-fluid layer is conveyed implicitly to \eqref{LowSys1} by way of the dispersion relations \eqref{A} and \eqref{eqA}. As we are examining travelling waves, we can transform to a moving frame where the motion is steady by way of the change of variables
\begin{equation}\label{cotrans} 
X(t)=kx(t)-\omega t,\qquad Y(t)=k(y(t)+h),\end{equation}
(recall $c=\omega/k$) which transforms system \eqref{LowSys1} to the autonomous system
\begin{subequations}\label{LowSys2}
 \begin{align}[left ={\empheqlbrace}]
 \label{LowSys2-a}     
 \frac{dX}{dt}&=M\cos (X)\cosh(Y)-\omega,  \\  
 \label{LowSys2-b} \frac{dY}{dt}&=M\sin (X)\sinh(Y),
\end{align}
\end{subequations}
with $(X(0),Y(0))=(x_0,y_0)$, and where we denote
\begin{equation}\label{M}
M:=\frac{a k \omega }{\sinh (kh)}=\frac{a}{h}\cdot \frac{kh}{\sinh (kh)}\cdot \omega\ll \omega, \end{equation}
since $s<\sinh(s)$ for $s>0$ and  $a/h \sim \mathcal O(\epsilon) \ll1$ in the linear wave regime.
Since \eqref{LowSys2} is periodic in $X$   we focus on the strip $\{X:-\pi\leq X\leq\pi\}$,
and the change of variables \eqref{cotrans} transforms the lower-fluid layer to the region 
$\{Y:  0\leq Y\leq kh+\mathfrak e \cos(X)\}$, where we denote by $\mathfrak e=a k$  the {\em wave-steepness} parameter for the internal interface. Without loss of generality, we choose $a>0$ (hence $M>0$) throughout this subsection (a change in the sign of $M$  can be effectuated in \eqref{LowSys2} by simply shifting the $X$ variable by $\pi$: from \eqref{eta-ans}, the difference between choice of signs corresponds physically to choosing either the crest, or the trough, of the internal wave to be located at $X=0$, respectively).

The autonomous system \eqref{LowSys2} meets standard regularity assumptions for the uniqueness of the Cauchy problem \cite{Meiss}, therefore  the existence of unique local smooth solutions is  ensured by the Picard--Lindel\"of theorem \cite{Meiss}, and the trajectories of \eqref{LowSys2}  do not intersect.
The right-hand side of \eqref{LowSys2-a} is an even function in both $X$ and $Y$, while the right-hand side of \eqref{LowSys2-b} is an odd function of both $X$ and $Y$: therefore the trajectories of \eqref{LowSys2} have mirror symmetries with respect to both the $X-$ and $Y-$axes.
The dynamical system \eqref{LowSys2} has the standard (canonical) form  \cite{Meiss} of Hamilton's equations
\[
\dot X= \frac{\partial H}{\partial Y}, \qquad \dot Y = -\frac{\partial H}{\partial X}
\] 
 for the Hamiltonian function $H(X,Y)=M \sinh{Y}\cos X -\omega Y$. Hence, if  $\left(X(t),Y(t)\right)$ is a trajectory of \eqref{LowSys2}, then we have $\frac{d }{dt} H (X(t),Y(t))=\frac{\partial H}{\partial X} \dot X + \frac{\partial H}{\partial Y} \dot Y =-\dot Y \dot X + \dot X \dot Y=0$, and so singular points of \eqref{LowSys2}  correspond to critical points of $H$. The nature of these points can be determined from the Hessian of $H$, and the associated Morse index \cite{Mil}, which corresponds to the number of negative eigenvalues of the Hessian. 
 
 The $0$-isocline is defined to be the set where $dY/dt=0$, and the $\infty$-isocline is the set where $dX/dt=0$, and singular points for the system \eqref{LowSys2} occur if these two sets have a non-empty intersection.
It can be shown (see \cite{HV-JDE}) that the system \eqref{LowSys2} has one singular point $Q=(0,Y^*)$, where $Y^*=\cosh^{-1}\left({\omega}/{M}\right)$, which is a saddle point, and the phase portrait for the lower-fluid layer is given in Figure \ref{figPP-low}.
\begin{figure}[h!]
\begin{center}
 \resizebox{.6\textwidth}{!}{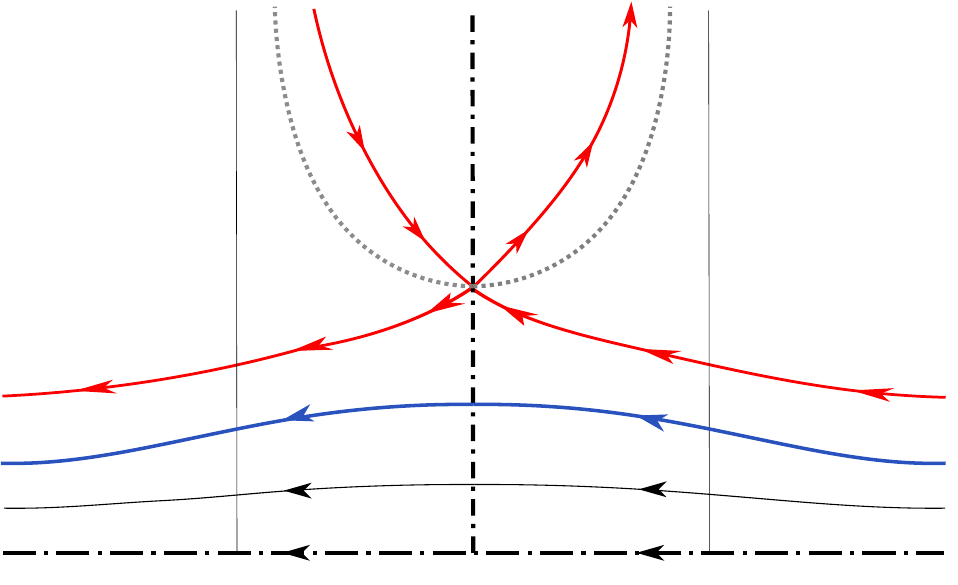}
\vspace*{5mm} 
 \caption[Cap]{
Phase portrait for the lower-fluid layer. The dotted grey line \tikz[baseline=-0.5ex]\draw[color=gray, very thick, densely dotted] (0,0) -- (.4,0); represents the $\infty-$isocline, with the dotted-dashed lines \tikz[baseline=-0.5ex]\draw[thick,dash dot] (0,0) -- (.5,0);  representing the $0-$isoclines. The internal wave profile (\tikz[baseline=-0.5ex] \draw[color=bleudefrance, very thick] (.0,0) -- (.7,0);) with mean-water level $Y=kh$ (corresponding to $y=0$) is also illustrated.
}
\label{figPP-low}
\end{center}
\end{figure}

\subsection{Upper-fluid layer}\label{secPP1}
Transforming to the moving frame by the change of variables
\begin{equation}\label{cotrans1} 
X(t)=kx(t)-\omega t,\qquad Y_1(t)=k(h_1-y(t)),\end{equation}
transforms system \eqref{UpSys1} to the autonomous system
\begin{align}[left ={\empheqlbrace}]
\label{UpSys2}
\frac{dX}{dt}&=F(X,Y_1):=M_1A\cos(X)\cosh (Y_1)-M_1\cos(X)\sinh (Y_1)-\omega
\\
\nonumber 
\frac{dY_1}{dt}&=G(X,Y_1):=M_1A\sin(X)\sinh (Y_1)-M_1\sin(X)\cosh (Y_1),
\end{align}
where we denote the parameter $M_1=a_1k\omega$.  Without loss of generality, we fix $M_1>0$ (and hence $a_1>0$) throughout subsequent considerations, with the sign of $a$ free to change (in order to match that of the ratio $a/a_1$ which is in turn prescribed by the value of $A$  via \eqref{A-a}). The sign of the ratio $a/a_1$ determines whether the surface and internal waves are coupled in-phase, or out-of-phase, a key characteristic concerning the qualitative nature of system \eqref{UpSys2}.
\begin{figure}[h!]
\begin{center}
 \hfill
    \subfloat[In-phase waves. \label{subfig-3b}]{%
      \resizebox{.33\textwidth}{!}{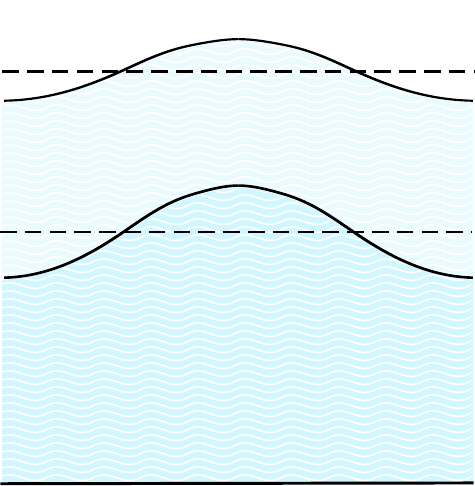}
    }
   \hfill
    \subfloat[Out-of-phase waves.  \label{subfig-3a}]{%
     \resizebox{.33\textwidth}{!}{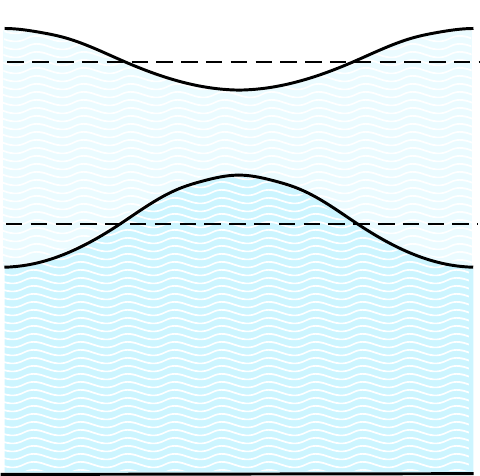}
    }
     \hfill
    \label{Fig3}
    \end{center}
    \caption{Coupled ``in phase'' and ``out-of-phase'' surface and internal water waves.}
  \end{figure}
The change of variables \eqref{cotrans1} reflects vertical coordinates through the line $y=h_1$ with the effect that wave crests, in terms of physical $(x,y)$ variables,   correspond to troughs when represented in terms of $(X,Y_1)$ variables, and vice-versa. Additionally, when expressed in terms of  $(X,Y_1)$ variables,  the streamline corresponding to the surface wave lies {\em beneath} that of the internal wave.
It is expedient to further re-express the right-hand sides of \eqref{UpSys2} as
\begin{align}   \label{FGa}
F(X,Y_1)=M_1\cos(X)f(Y_1)-\omega,  \quad  G(X,Y_1)=M_1\sin(X)g(Y_1),
\end{align}
for the functions
\begin{equation}   \nonumber   \label{FGb}
 f(Y_1):=A\cosh (Y_1)-\sinh (Y_1), \  g(Y_1):=A\sinh (Y_1)-\cosh (Y_1). 
\end{equation}
Note that $f'(Y_1)=g(Y_1)$, and $g'(Y_1)=f(Y_1)$. A key component in the construction of phase portraits for system \eqref{UpSys2} is the determination of the $0$-isocline (the set of points where the vector field is horizontal: $\dot Y_1=G(X,Y_1)=0$ in \eqref{FGa}) and the $\infty$-isocline (the set of points where the vector field is vertical: $\dot X=F(X,Y_1)=0$ in \eqref{FGa}). 
System \eqref{UpSys2}  is also a Hamiltonian system for which
\begin{equation}   \nonumber  \label{HAM1}
\dot X=F(X,Y_1)= \frac{\partial H_1}{\partial Y_1}, \qquad \dot Y_1 =G(X,Y_1)= -\frac{\partial H_1}{\partial X},
\end{equation}
where the Hamiltonian function is defined $H_1(X,Y_1) = M_1\cos{X}g(Y_1) - kc Y_1$.
Note that $H_1$ is constant along trajectories of \eqref{UpSys2}, and singular points of \eqref{UpSys2} correspond to critical points of $H_1(X,Y_1)$.

As \eqref{UpSys2} is periodic in $X$   we need only consider the strip $\{X:-\pi\leq X\leq\pi\}$
and, due to the definition of the $Y_1$ in \eqref{cotrans1}, the physically--relevant  solutions of \eqref{UpSys2}  will be located in the region $\{Y_1:  -\mathfrak e_1 \cos \left( X \right) \leq Y_1 \leq kh_1-\mathfrak  e \cos \left(X\right)\}$, where we denote by $\mathfrak e_1=a_1 k=2\pi \cdot a_1/\lambda$ the {\em wave-steepness} parameter for the surface. This non-dimensional parameter can be expressed $\mathfrak e_1=2\pi \cdot  \delta \cdot \epsilon_1$ in terms of the  wave-amplitude parameter  $\epsilon_1$,  and shallowness parameter  $\delta$: accordingly $\mathfrak e_1\ll 1$. We note that $M_1\ll \omega$ since $M_1/\omega=\mathfrak e_1$.
The autonomous system \eqref{UpSys2} meets the standard regularity assumptions for the uniqueness of the Cauchy problem \cite{Meiss}, therefore its trajectories do not intersect. Moreover, since $F(X,Y_1)$ is an even function, and  $G(X,Y_1)$ an odd function, with respect to $X$, any trajectory of system \eqref{UpSys2} is symmetric with respect to the $Y_1$-axis when viewed as a curve in the $(X,Y_1)-$phase plane.


\subsubsection{System \eqref{UpSys2} with $A=1$:} \label{ppA=1}
For $A=1$, we have $f(Y_1)=e^{-Y_1}=-g(Y_1)$ and the phase portrait for positive values of $Y_1$ rapidly converges to a series of flat, horizontal lines as $Y_1$ increases, since the velocity field converges exponentially fast to the uniform system $\dot{X}\equiv-\omega$ and $\dot{Y_1}\equiv0$. The point $Q_1=(0,Y_1^*)$ is a singular point, for $Y_1^*=-\ln\left(1/\mathfrak e_1\right)$, and it is a saddle point. The phase portrait for system \eqref{UpSys2} when $A=1$ is given in Figure \ref{figPP-up2}.
\begin{figure}[h!]
\begin{center}
 \resizebox{.6\textwidth}{!}{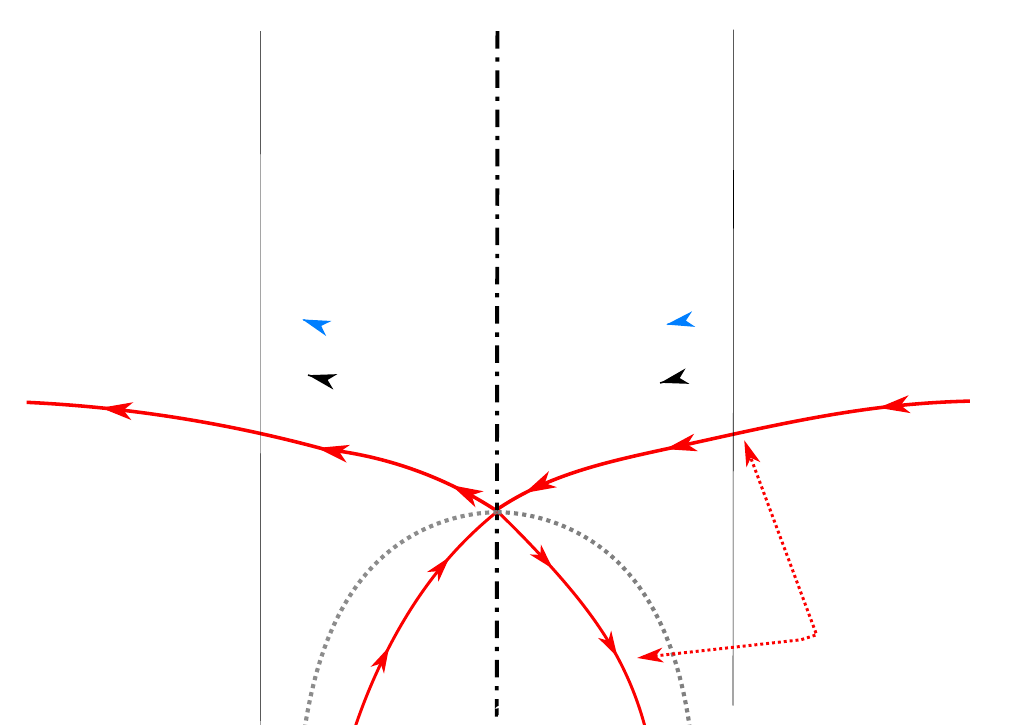}
  \caption[]{
Phase portrait of the upper-fluid layer when $A=1$. The dotted grey line \tikz[baseline=-0.5ex] \draw[color=gray, very thick, densely dotted] (.0,0) -- (.4,0) ; represents the $\infty-$isocline, while the dotted-dashed lines \tikz[baseline=-0.5ex] \draw[thick,dash dot] (.0,0) -- (.5,0) ;  represent the $0-$isoclines.  The {surface wave} profile (\tikz[baseline=-0.5ex] \draw[color=azure, very thick] (.0,0) -- (.7,0);) has mean-water level $Y_1=0$, corresponding to $y=h_1$. The internal wave profile (\tikz[baseline=-0.5ex] \draw[color=bleudefrance, very thick] (.0,0) -- (.7,0);) has mean water level $Y=kh_1$, corresponding to $y=0$.
}
\label{figPP-up2}
\end{center}
\end{figure}
We note that the the surface-wave streamline  must be located above the separatrix for physically--relevant solutions of \eqref{UpSys2}, and this is the case for  $-\mathfrak e_1>Y_1^*$, that is, 
\begin{equation}  \nonumber   
\mathfrak e_1 e^{\mathfrak e_1}<1.
\end{equation}


\subsubsection{System \eqref{UpSys2} with $A>1$:}\label{ppA>1}
There is a unique solution $\bar{Y_1}$ to $g({Y_1})=0$ for $A>1$ with \begin{equation}   \nonumber  
\bar{Y_1}=\frac{1}{2}\ln\left(\frac{A+1}{A-1}\right),
\end{equation} 
which is well-defined, and corresponds to $A=\coth(\bar{Y_1})$. It follows that $A \rightarrow 1 \Leftrightarrow \bar{Y_1} \rightarrow \infty$, while $A \rightarrow \infty  \Leftrightarrow \bar{Y_1} \rightarrow 0$.  
 Hence, the $0-$isocline is composed of  the vertical half-lines $X=0$, $X=\pi$, and the horizontal line-segment $Y_1=\bar{Y_1}$.   
The study of the $\infty$-isocline, where $\dot X=F(X,Y_1)=0$, can be achieved through examining $f(Y_1)$, a schematic for which  is given in Figure \ref{figf(Y)}.
\begin{figure}[h!]
\begin{center}
 \resizebox{.6\textwidth}{!}{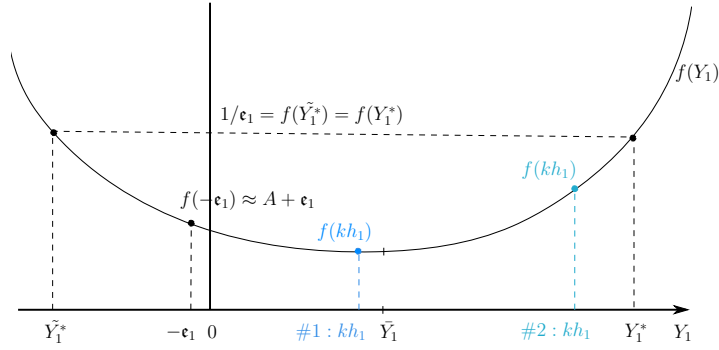}
\caption{
Schematic of $f(Y_1)$ for $A>1$, where $f(0)=A$, with $A=\cosh(\bar{Y_1})/\sinh(\bar{Y_1})$, and the minimum value attained is $f(\bar{Y_1})=1/\sinh(\bar{Y_1})$.
}\label{figf(Y)}
\end{center}
\end{figure}
If $\mathfrak e_1<\max_{Y_1} 1/f(Y_1)=\sinh(\bar{Y_1})$, then there exists a pair of values, $\tilde{Y^*_1}$, $Y^*_1$ with $\tilde{Y^*_1}\leq Y^*_1$, say, such that $f(\tilde{Y^*_1})=f(Y^*_1)=1/\mathfrak e_1$. The system \eqref{UpSys2} then has singular points at $\tilde{Q_1}=(\pi,\tilde {Y_1^*})$, $Q_1=(\pi,Y_1^*)$, both of whom are saddle points. The phase portrait of system \eqref{UpSys2} is given in Figure \ref{figPP-up3}.
\begin{figure}[h!]
\begin{center}
 \resizebox{.6\textwidth}{!}{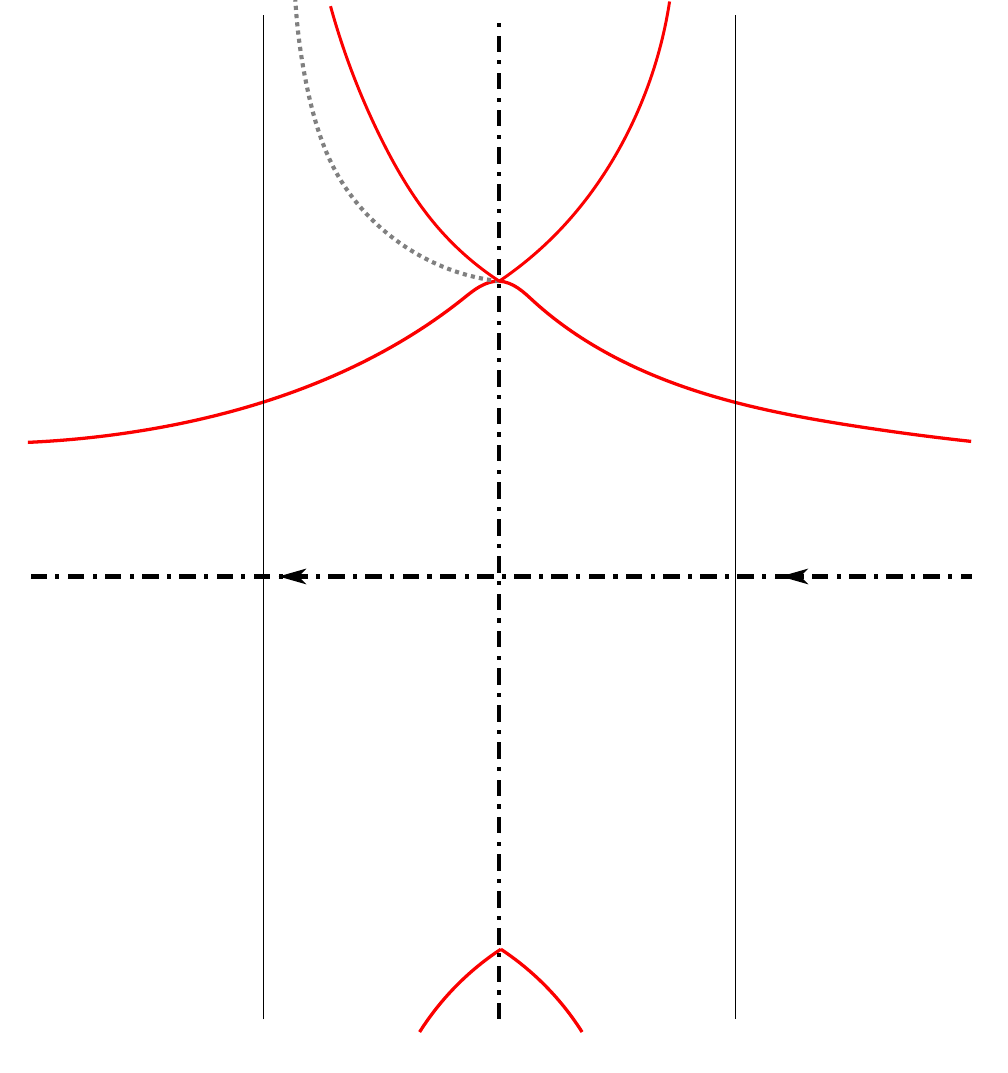}
  \caption[]{
Phase portrait for the upper-fluid layer when $A>1$. The dotted grey lines \tikz[baseline=-0.5ex] \draw[color=gray, very thick, densely dotted] (.0,0) -- (.4,0) ; represent the $\infty-$isoclines, with the dotted-dashed lines \tikz[baseline=-0.5ex] \draw[thick,dash dot] (.0,0) -- (.5,0) ;  representing the $0-$isoclines. The {surface wave} profile (\tikz[baseline=-0.5ex] \draw[color=azure, very thick] (.0,0) -- (.7,0);) has mean-water level $Y_1=0$, corresponding to $y=h_1$. The internal wave profile is illustrated for two differing values of the mean water level $Y=kh_1$, corresponding to $y=0$: $kh_1 \leq \bar{Y_1}$ in Case 1 (\tikz[baseline=-0.5ex] \draw[color=bleudefrance, dashed, dash pattern=on 8pt off 3pt, very thick] (.0,0) -- (.7,0);), whereas  {$kh_1 > \bar{Y_1}$} for Case  2 (\tikz[baseline=-0.5ex] \draw[color=ballblue, dashed, dash pattern=on 4pt off 1pt, very thick] (.0,0) -- (.7,0);).}
\label{figPP-up3}
\end{center}
\end{figure}

The surface wave profile has a mean-water level $Y_1=0$, which corresponds to $y=h_1$, while the internal wave profile has a mean water level $Y=kh_1$, which corresponds to $y=0$. From the phase portrait in Figure \ref{figPP-up3} we see that there are two qualitatively different fluid motions possible in the upper-fluid layer, which depends on the location of $Y=kh_1$.

In Case 1, {$kh_1 \leq \bar{Y_1}$}  and the internal wave profile is in-phase with the surface wave: in this case, all streamlines have their crests located at $X=0$.  If {$kh_1=\bar{Y_1}$}  then the internal interface is a flat horizontal line. Whenever $kh_1 \leq \bar{Y_1}$, there is an additional bound on $\mathfrak e$, namely $\mathfrak e <kh_1-\bar{Y_1}$.
In Case 2, $kh_1 > \bar{Y_1}$ and the internal wave is now out-of-phase with the surface wave. In this scenario, all streamlines beneath the line $Y_1=\bar{Y_1}$ have their crest located at $X=0$, with the amplitudes diminishing steadily until they vanish at the $0-$isocline $Y_1=\bar{Y_1}$.
As we move above  $Y_1=\bar{Y_1}$ the amplitude of the streamlines increase steadily, until the reach a maximum at $Y_1=kh_1$, which corresponds to the internal wave profile. However, the vertical velocity reverses direction as we pass $Y_1=\bar{Y_1}$, and so the line $X=0$ now corresponds to the wave trough.

In order for system \eqref{UpSys2} to describe physically--relevant solutions, the singular points $\tilde{Q_1}$ and $Q_1$ must lie outside the upper-fluid layer, hence we must have $\tilde {Y_1^*}<-\mathfrak e_1$ and $kh_1 +\mathfrak e<Y_1^*$ or, equivalently (cf. Figure \ref{figf(Y)}),
\begin{equation}   \nonumber  
\max\left\{ f(-\mathfrak e_1), f(kh_1 +\mathfrak e) \right\}<1/\mathfrak e_1.
\end{equation}


\section{Particle trajectories}\label{TRAJ}
The phase-plane analysis of  Section \ref{secPP} furnishes us with an Eulerian description of the fluid motion, in the sense that the phase portraits in  Figures \ref{figPP-low}, \ref{figPP-up2}, \ref{figPP-up3} give us a qualitative depiction of the fluid streamlines in both the lower, and upper, fluid layers as viewed by a fixed observer in the $(X,Y)$, and $(X,Y_1)$, reference frames, respectively.
In this section a Lagrangian description of the fluid motion is provided, whereby the particle trajectories  traced by any fixed fluid particle  as time evolves are ascertained. Particle trajectories $(x(t),y(t))$ are prescribed by  nonautonomous systems  \eqref{LowSys1}, and \eqref{UpSys1}. The fact that we are dealing with nonautonomous motion complicates matters significantly, however we may use information from the autonomous systems \eqref{LowSys2}, and \eqref{UpSys2}, to infer a physical description of fluid motion by way of reversing transformations \eqref{cotrans}, and \eqref{cotrans1}, through
\begin{equation}
\label{revtrans} 
x(t)=\frac{X(t)}{k}+ct, \quad y(t)=\frac{Y(t)}{k}-h,
\end{equation}
in the lower-fluid layer, and 
\begin{equation}
\label{revtrans1} 
x(t)=\frac{X(t)}{k}+ct, \quad y(t)=h_1-\frac{Y_1(t)}{k},
\end{equation}
in the upper-fluid layer, respectively.  Note that the  reflection term  involved in the vertical coordinate transformation in \eqref{revtrans1} reverses the  vertical orientation of fluid motion when expressed in terms of the physical  $(x,y)$ coordinates, as opposed to the $(X,Y_1)$ coordinates.

One key question that may be addressed is whether, or not, fluid particles describe closed trajectories. The following lemma, which is stated for motion in the lower fluid layer, but applies equally  to fluid motion in  the upper layer, plays a central role in answering this question. Note that this result is independent of the linearity of solutions, and relies solely on periodicity considerations. Indeed, similar results play a key role in the analysis of particle trajectories beneath nonlinear Stokes waves \cite{Con-06,ConStr,Hen-06,Hen-08}, and we note that some patterns for Stokes waves resemble those found here \cite{IK-08}.
 Suppose  $(X(t),Y(t))$ describes a streamline in the lower-fluid layer such that $(X(0),Y(0))=(\pi,Y^0)$, and let $t_{Y^0}(-\pi)$ denote the time taken for the particle to intersect the line $X=-\pi$ (note that  $\dot X<0$ uniformly along fluid streamlines in the moving frames).
\begin{lemma}\label{period}
If the particle trajectory prescribed by $(x(t),y(t))$ is a closed path with period $\tau$ then, necessarily, we have $\tau=\frac{2\pi}{\omega}$. Conversely, suppose $t_{Y^0}(-\pi)=\frac{2\pi}{\omega}$, then the particle path prescribed by $(x(t),y(t))$ is closed.
\end{lemma}
\begin{proof}
The proof follows from the periodicity of  the autonomous systems   \eqref{LowSys2}, or \eqref{UpSys2}, with respect to $X$, together with the definitions \eqref{revtrans} and \eqref{revtrans1}.
\end{proof}
Although they follow directly as in \cite{HV-JDE},  we include the proofs of Theorem \ref{nopersol} and Proposition \ref{propDrift} in the lower-fluid layer setting of Section \ref{secLLL} below, for completeness, and to illustrate the approach involved in establishing them. 
We will omit the corresponding proofs for the more convoluted setting of the upper fluid layer in Section \ref{secULL}, and refer the reader to \cite{HV-JDE} for details.

\subsection{Lower-fluid layer}\label{secLLL}
In the lower-fluid layer, focus is restricted to physically admissible streamlines for which $(X(0),Y(0))=(\pi,Y^0)$,  with $Y^0\in[0,k(h-a)]$. 
\begin{theorem}
\label{nopersol}
The system \eqref{LowSys1} has no  solutions $(x(t),y(t))$ which are periodic. Accordingly, there are no closed particle paths in the lower-fluid layer, instead all fluid particles experience a positive horizontal drift.
\end{theorem}
\begin{proof} 
Bearing in mind Lemma \ref{period}, it suffices to show that $t_{Y^0}(-\pi)>\frac{2\pi}{\omega}$ in order to prove the theorem.
We start with the case $Y^0=0$, where the streamline is located on the flat bed and so $Y(t)\equiv0$ and $X(t)$ can be explicitly obtained by solving $\dot X=M\cos(X)-\omega$, with $X(0)=\pi$.
It follows that 
\begin{equation}
\label{tbed} 
t_{0}(-\pi)=\int^{\pi}_{-\pi}\frac{ds}{\omega-M\cos(s)}={2\pi}\sqrt{\frac{1}{\omega^2-M^2}}>\frac{2\pi}{\omega}, 
\end{equation}
where we use the fact that $M<\omega$ (by \eqref{M})  and   the integral   \[
\int^z_0\frac{ds}{\omega-M\cos(s)}={2}\sqrt{\frac{1}{\omega^2-M^2}}\arctan\left(\sqrt{\frac{\omega+M}{\omega-M}}\tan\left(\frac{z}{2}\right) \right), \quad z>0.
\]
For  $Y^0\in[0,kh- \mathfrak e]$, $dY/dt>0$ for $X\in (0,\pi)$, and $dY/dt<0$ when $X\in (-\pi,0)$. If   this streamline  intersects the line $X=\frac{\pi}{2}$ at the value $Y=\mathcal Y$, then $(X(t),Y(t))$ lies below the line $Y=\mathcal Y$ for $X(t)\in [-\pi, -\frac{\pi}{2})\cup(\frac{\pi}{2},\pi]$, and lies above the line for $X(t)\in (-\frac{\pi}{2},\frac{\pi}{2})$. Thus,
\begin{equation}\label{x>barx} 
\dot X=M\cosh(Y)\cos(X)-\omega\geq M\cosh(\mathcal Y)\cos(X)-\omega, \quad t\geq 0.
\end{equation}
Introducing the differential equation  $\dot{\mathfrak X}=M\cosh (\mathcal Y)\cos (\mathfrak X)-\omega$,
 with $\mathfrak X (0)=\pi$, 
it follows immediately from \eqref{x>barx} and the fact that $X(0)=\mathfrak X(0)=\pi$ that $X(t)\geq \mathfrak X(t)$ for $t\geq 0$. Therefore $t_{Y^0}(-\pi)> t^*$, where $t^*$ is the time it takes for $\mathfrak X (t^*)=-\pi$.
In a manner similar to solving \eqref{tbed}, the value of $t^*$ can be explicitly computed as being 
\[
t^*={2\pi}\sqrt{\frac{1}{\omega^2-M^2\cosh^2(\mathcal Y)}}>\frac{2\pi}{\omega},
\]
and hence $t_{Y^0}(-\pi)>\frac{2\pi}{\omega}$. This completes the proof.
\end{proof}
The analysis of system \eqref{LowSys2} undertaken in Section \ref{secPP}, coupled with \eqref{revtrans} and Theorem \ref{nopersol}, facilitates a qualitative description of physical particle motion  in the lower-fluid layer as  prescribed by \eqref{LowSys1}. Assume a fluid particle is initially at its greatest depth $y(0)=y_0$: we label this position $A$.  This corresponds to $X(0)=\pi$, and since $\dot X<0$ along streamlines it follows that,   in the moving frame, the variable $X(t)$ decreases continuously from $\pi$ to $-\pi$, and we have:  $\dot x <0$, $\dot y >0$ for $X(t)\in(\pi/2,\pi)$; $\dot x>0$, $\dot y >0$ for $X(t)\in(0,\pi/2)$;   $\dot x >0$, $\dot y<0$ for $X(t)\in(-\pi/2,0)$;  $\dot x <0$, $\dot y<0$ for $X(t)\in(-\pi,-\pi/2)$.  The particle  returns to its lowest position in the fluid layer (with depth $y=y_0$) at time $t=t_{Y^0}(-\pi)>{2\pi}/{\omega}$, having experienced a positive horizontal drift: we label this position $B$.
\begin{figure}[H]
\begin{center}
 \resizebox{.32\textwidth}{!}{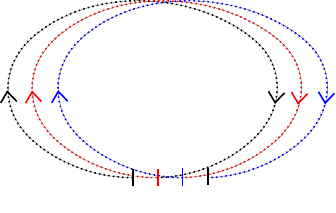}
  \caption[]{Schematic of a typical particle trajectory in the lower-fluid layer, representing its location at time: $t=0$ ($A$);  $t=t_{Y_0}(-\pi)$ ($B$); $t=2t_{Y_0}(-\pi)$ ($C$);  $t=3t_{Y_0}(-\pi)$ ($D$).
}
\label{figPartLow}
\end{center}
\end{figure}
\begin{proposition}\label{propDrift}
The horizontal drift experienced by a fluid particle in the lower fluid layer over one wave period decreases strictly with depth.
\end{proposition}
\begin{proof}
The horizontal drift experienced by a fluid particle over one wave period is given by $D(Y^0)=x(t_{Y^0}(-\pi))-x(\pi)=(\omega t_{Y^0}(-\pi)-2\pi)/k>0$, which can be expressed from \eqref{LowSys2-a}  as
\begin{equation*}
D(Y^0)=2\int_{0}^{\pi}\frac{M\cosh(Y)\cos(X) dX}{\omega-M\cosh(Y)\cos(X)}
\end{equation*}
Let $Y^1<Y^0$, with $\tilde Y=\tilde Y(X)$ denoting the streamline for which $\tilde Y(\pi)=Y^1$.
\begin{align*}
D(Y^0)-D(Y^1)=2\int_{0}^{\pi}\!\!\frac{\omega M\left(\cosh(Y)-\cosh(\tilde Y)\right)\cos(X)}{\left(\omega-M\cosh(Y)\cos(X)\right)\left(\omega-M\cosh(\tilde Y)\cos(X)\right)}dX,
\end{align*}
which, in the limit $Y^1\to Y^0$, has the same sign as 
\begin{align*}
\lim_{\tilde Y\to Y} 2\int_{0}^{\pi}\frac{\omega M\left(\cosh(Y)-\cosh(\tilde Y)\right)\cos(X)}{\left(\omega-M\cosh(Y)\cos(X)\right)\left(\omega-M\cosh(\tilde Y)\cos(X)\right)(Y-\tilde Y)}dX
\\
=2\int_{0}^{\pi}\frac{\omega M\cosh(Y)\cos(X)}{\left(\omega-M\cosh(Y)\cos(X)\right)^2}dX
>2\int_{0}^{\pi}\frac{\omega M\cosh(Y)\cos(X)}{\omega^2}dX
\\
>2\int_{0}^{\pi}\frac{M}{\omega}\cosh(\mathcal Y)\cos(X)dX=0.
\end{align*}
Thus $D(Y^0)>D(Y^1)$, and the particle drift is decreasing with depth.
\end{proof}

\subsection{Upper-fluid layer}\label{secULL}
In the upper-fluid layer, focus is restricted to physically admissible streamlines for which $(X(0),Y_1(0))=(\pi,Y_1^0)$,  with $Y_1^0\in[\mathfrak e_1,kh_1+\mathfrak e]$ if $a/a_1>0$, whereas $Y_1^0\in[\mathfrak e_1,kh_1-\mathfrak e]$ if $a/a_1<0$.
The horizontal drift $D(Y_1^0)=x(t_{Y_1^0}(-\pi))-x(\pi)$ experienced by a fluid particle over one wave period in the upper fluid layer is given by
\begin{equation*}
D(Y_1^0)=2\int_{0}^{\pi}\frac{M_1f(Y_1)\cos(X) dX}{\omega-M_1f(Y_1)\cos(X)}.
\end{equation*}
If $Y_1^1<Y_1^0$ then, in the limit $Y_1^1\to Y_1^0$, the difference between the drifts $D(Y_1^0)-D(Y_1^1)$ has the same sign as 
\begin{align}
2\int_{0}^{\pi}\!\!\!\frac{\omega M_1g(Y_1)\cos(X)}{\left(\omega-M_1f(Y_1)\cos(X)\right)^2}dX.
\label{DDD}
\end{align}
As with the phase-plane analysis of Section \ref{secPP1}, the qualitative behaviour of fluid particles in the upper-layer is markedly different  depending on the  parameter values $A=1$, $A>1$: we can again rule out the possibility $A<1$ (considered in \cite{HV-JDE,HV-AN}) due to Theorem \ref{THM}. We deal with these cases separately. 
\subsubsection{$A> 1$}
\begin{theorem}
\label{nopersol1}
Let $A> 1$. There are no closed particle paths in the upper-fluid layer whose motion is governed by the system \eqref{UpSys1}. That is, the system \eqref{UpSys1} has no  solutions $(x(t),y(t))$ which are periodic.
\end{theorem}
The qualitative nature of fluid particle trajectories in the upper-fluid layer prescribed by \eqref{UpSys1} when $A>1$ can now be determined  from the analysis of system \eqref{UpSys2} performed in Section \ref{secPP1}, coupled with transformation \eqref{revtrans1} and Theorem \ref{nopersol1}. 
 For the sake of generality, we assume that $kh_1>\bar{Y_1}$: if the alternative is true, simply restrict attention to streamlines with $Y^0_1\in[\mathfrak e_1,\bar{Y_1})$ in this discussion.  For streamlines in this region  (which corresponds physically to the top section of the upper-fluid layer) the particle motion is captured in schematic $(a)$ of Figure \ref{figPartUp1}. 

At the $0-$isocline $Y=\bar{Y_1}$  fluid particles experience a forward horizontal drift $x(t_{\bar{Y_1}}(-\pi))-x(0)=\left({t_{\bar{Y_1}}(-\pi)\omega-2\pi}\right)/{k}>0$.
This motion is represented by schematic $(b)$ of Figure \ref{figPartUp1}. 

 For streamlines with  $Y_1^0\in (\bar{Y_1}, kh_1+\mathfrak  e]$ (which corresponds physically to the bottom region of the upper-fluid layer) particle motion is depicted in schematic $(c)$ of Figure \ref{figPartUp1}. 
\begin{figure}[h!]
\begin{center}
 \resizebox{.32\textwidth}{!}{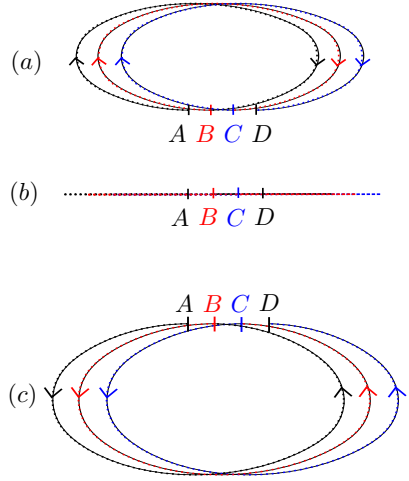}
  \caption[]{Schematics of typical  trajectories for particles $(a)$, $(b)$, and $(c)$ in the upper-fluid layer, for $A>1$ and assuming that  $kh_1>\bar{Y_1}$. Their location  at times $t=0,t_{Y_1^0}(-\pi),2t_{Y_1^0}(-\pi),3t_{Y_1^0}(-\pi)$ are denoted by $A,B,C,D$, respectively.
}
\label{figPartUp1}
\end{center}
\end{figure}

\begin{proposition}\label{DDE}
Suppose $A>1$. The horizontal drift experienced by a fluid particle in the upper fluid layer over one wave period decreases with depth for  fluid motion of the type illustrated in schematic (a) of Figure \ref{figPartUp1}, while it increases with depth for  fluid motion of the type illustrated in schematic (c) of Figure \ref{figPartUp1}. Accordingly, the minimum horizontal drift experienced by fluid particles occurs at $Y=\bar{Y_1}$, as illustrated in schematic (b) of Figure \ref{figPartUp1}.
\end{proposition}
The proof follows from expression \eqref{DDD}, and monotonicity properties of $f(Y_1)$, $g(Y_1)$. We refer to \cite{HV-JDE} for full details.

\subsubsection{$A=1$}
Due  to the straightforward form that the functions $f(Y_1),g(Y_1)$ assume in the setting $A=1$, a comprehensive qualitative description of fluid motion  can be achieved.  
\begin{theorem}
\label{nopersol2}
Let $A=1$. There are no closed particle paths in the upper-fluid layer whose motion is governed by the system \eqref{UpSys1}. 
\end{theorem}
\begin{proof}
For $A=1$ we have $f(Y_1)=e^{-Y_1}=-g(Y_1)$, and it can easily be shown that the analogue of inequality \eqref{x>barx} holds for streamlines in the region $Y_1^0\in [\mathfrak e_1, kh_1+\mathfrak  e]$. Consequently,  
$
t_{Y_1^0}(-\pi)>{2\pi}\sqrt{\frac{1}{\omega^2-M_1^2f^2(\mathcal Y_1)}}>\frac{2\pi}{\omega}.
$
\end{proof}
Fluid particle motion matches that illustrated in schematic $(a)$ of Figure \ref{figPartUp1}, with all fluid particles experiencing a forward drift, and accordingly the following holds.

\begin{proposition}
Suppose $A=1$. The horizontal drift experienced by a fluid particle in the upper fluid layer over one wave period is decreasing with depth.
\end{proposition}

\section*{Acknowledgements}
DH and RI$^2$ would like to acknowledge the support of {\em Research Ireland} under grant number 21/FFP-A/9150. The authors would like to thank the referees for their constructive suggestions.

\end{document}

%% file: Fig1.pdf_tex
\begingroup%
  \makeatletter%
  \providecommand\color[2][]{%
    \errmessage{(Inkscape) Color is used for the text in Inkscape, but the package 'color.sty' is not loaded}%
    \renewcommand\color[2][]{}%
  }%
  \providecommand\transparent[1]{%
    \errmessage{(Inkscape) Transparency is used (non-zero) for the text in Inkscape, but the package 'transparent.sty' is not loaded}%
    \renewcommand\transparent[1]{}%
  }%
  \providecommand\rotatebox[2]{#2}%
  \newcommand*\fsize{\dimexpr\f@size pt\relax}%
  \newcommand*\lineheight[1]{\fontsize{\fsize}{#1\fsize}\selectfont}%
  \ifx\svgwidth\undefined%
    \setlength{\unitlength}{465.62735971bp}%
    \ifx\svgscale\undefined%
      \relax%
    \else%
      \setlength{\unitlength}{\unitlength * \real{\svgscale}}%
    \fi%
  \else%
    \setlength{\unitlength}{\svgwidth}%
  \fi%
  \global\let\svgwidth\undefined%
  \global\let\svgscale\undefined%
  \makeatother%
  \begin{picture}(1,0.51469493)%
    \lineheight{1}%
    \setlength\tabcolsep{0pt}%
    \put(0,0){\includegraphics[width=\unitlength,page=1]{Fig1.pdf}}%
    \put(0.19011275,0.49160867){\makebox(0,0)[lt]{\lineheight{0}\smash{\begin{tabular}[t]{l}{\large $y=h_1+\eta_1$}\end{tabular}}}}%
    \put(0.19004735,0.34991343){\makebox(0,0)[lt]{\lineheight{0}\smash{\begin{tabular}[t]{l}{\large $y=\eta$}\end{tabular}}}}%
    \put(0.9794442,0.44011431){\makebox(0,0)[lt]{\lineheight{0}\smash{\begin{tabular}[t]{l}$y=h_1$\end{tabular}}}}%
    \put(0.9794442,0.27256566){\makebox(0,0)[lt]{\lineheight{0}\smash{\begin{tabular}[t]{l}$y=0$\end{tabular}}}}%
    \put(0.98108765,0.01657416){\makebox(0,0)[lt]{\lineheight{0}\smash{\begin{tabular}[t]{l}$y=-h$\end{tabular}}}}%
    \put(0.65226988,0.38214441){\makebox(0,0)[lt]{\lineheight{0}\smash{\begin{tabular}[t]{l}{\large Density $\rho_1$}\end{tabular}}}}%
    \put(0.64398731,0.1551133){\makebox(0,0)[lt]{\lineheight{0}\smash{\begin{tabular}[t]{l}{\large Density $\rho=\rho_1(1+r)$}\end{tabular}}}}%
    \put(0,0){\includegraphics[width=\unitlength,page=2]{Fig1.pdf}}%
    \put(0.52038612,0.07768373){\makebox(0,0)[lt]{\lineheight{0}\smash{\begin{tabular}[t]{l}{\large $u$}\end{tabular}}}}%
    \put(0.41407793,0.19204556){\makebox(0,0)[lt]{\lineheight{0}\smash{\begin{tabular}[t]{l}{\large $v$}\end{tabular}}}}%
    \put(0.52199683,0.32412542){\makebox(0,0)[lt]{\lineheight{0}\smash{\begin{tabular}[t]{l}{\large $u_1$}\end{tabular}}}}%
    \put(0.39797066,0.38372247){\makebox(0,0)[lt]{\lineheight{0}\smash{\begin{tabular}[t]{l}{\large $v_1$}\end{tabular}}}}%
    \put(0.03904999,0.34222412){\makebox(0,0)[lt]{\lineheight{0}\smash{\begin{tabular}[t]{l}{\LARGE $\Omega_1$}\end{tabular}}}}%
    \put(0.04362377,0.12832999){\makebox(0,0)[lt]{\lineheight{0}\smash{\begin{tabular}[t]{l}{\LARGE $\Omega$}\end{tabular}}}}%
  \end{picture}%
\endgroup%

%% file: PP-Lower.pdf_tex
\begingroup%
  \makeatletter%
  \providecommand\color[2][]{%
    \errmessage{(Inkscape) Color is used for the text in Inkscape, but the package 'color.sty' is not loaded}%
    \renewcommand\color[2][]{}%
  }%
  \providecommand\transparent[1]{%
    \errmessage{(Inkscape) Transparency is used (non-zero) for the text in Inkscape, but the package 'transparent.sty' is not loaded}%
    \renewcommand\transparent[1]{}%
  }%
  \providecommand\rotatebox[2]{#2}%
  \newcommand*\fsize{\dimexpr\f@size pt\relax}%
  \newcommand*\lineheight[1]{\fontsize{\fsize}{#1\fsize}\selectfont}%
  \ifx\svgwidth\undefined%
    \setlength{\unitlength}{459.5260548bp}%
    \ifx\svgscale\undefined%
      \relax%
    \else%
      \setlength{\unitlength}{\unitlength * \real{\svgscale}}%
    \fi%
  \else%
    \setlength{\unitlength}{\svgwidth}%
  \fi%
  \global\let\svgwidth\undefined%
  \global\let\svgscale\undefined%
  \makeatother%
  \begin{picture}(1,0.61053104)%
    \lineheight{1}%
    \setlength\tabcolsep{0pt}%
    \put(0,0){\includegraphics[width=\unitlength,page=1]{PP-Lower.pdf}}%
    \put(0.44249232,-0.02724873){\makebox(0,0)[lt]{\lineheight{0}\smash{\begin{tabular}[t]{l}\Large{$X=0$}\end{tabular}}}}%
    \put(0.68241354,-0.02561659){\makebox(0,0)[lt]{\lineheight{0}\smash{\begin{tabular}[t]{l}\Large{$X=\frac{\pi}{2}$}\end{tabular}}}}%
    \put(0.92396671,-0.02398454){\makebox(0,0)[lt]{\lineheight{0}\smash{\begin{tabular}[t]{l}\Large{$X=\pi$}\end{tabular}}}}%
    \put(0.17972166,-0.02724864){\makebox(0,0)[lt]{\lineheight{0}\smash{\begin{tabular}[t]{l}\Large{$X=-\frac{\pi}{2}$}\end{tabular}}}}%
    \put(-0.03898196,-0.02724873){\makebox(0,0)[lt]{\lineheight{0}\smash{\begin{tabular}[t]{l}\Large{$X=-\pi$}\end{tabular}}}}%
    \put(-0.11405931,0.02824323){\makebox(0,0)[lt]{\lineheight{0}\smash{\begin{tabular}[t]{l}\Large{${Y=0}$}\end{tabular}}}}%
    \put(0,0){\includegraphics[width=\unitlength,page=2]{PP-Lower.pdf}}%
    \put(0.82767187,0.45259352){\color[rgb]{1,1,1}\makebox(0,0)[lt]{\lineheight{0}\smash{\begin{tabular}[t]{l}\Large{\textcolor{red}{Separatrix}}\end{tabular}}}}%
    \put(0.02349134,0.37832586){\makebox(0,0)[lt]{\lineheight{0}\smash{\begin{tabular}[t]{l}\Large{\textcolor{gray}{$\infty-$isocline}}\end{tabular}}}}%
    \put(0,0){\includegraphics[width=\unitlength,page=3]{PP-Lower.pdf}}%
    \put(1.00177647,0.03080063){\makebox(0,0)[lt]{\lineheight{0}\smash{\begin{tabular}[t]{l}\Large{${(y=-h)}$}\end{tabular}}}}%
    \put(0.52251083,0.30776718){\makebox(0,0)[lt]{\lineheight{0}\smash{\begin{tabular}[t]{l}\Large{$\textcolor{red}{Q=(0,Y^*)}$}\end{tabular}}}}%
    \put(0,0){\includegraphics[width=\unitlength,page=4]{PP-Lower.pdf}}%
    \put(-0.13585197,0.14731068){\makebox(0,0)[lt]{\lineheight{0}\smash{\begin{tabular}[t]{l}\textcolor{bleudefrance}{\Large{$Y=kh$}}\end{tabular}}}}%
    \put(1.0074611,0.14567854){\makebox(0,0)[lt]{\lineheight{0}\smash{\begin{tabular}[t]{l}\Large{$(y=0)$}\end{tabular}}}}%
    \put(-0.12430203,0.19188673){\makebox(0,0)[lt]{\lineheight{0}\smash{\begin{tabular}[t]{l}\Large{$\textcolor{red}{Y=\beta}$}\end{tabular}}}}%
    \put(0,0){\includegraphics[width=\unitlength,page=5]{PP-Lower.pdf}}%
  \end{picture}%
\endgroup%

%% file: Fig3-b.pdf_tex
\begingroup%
  \makeatletter%
  \providecommand\color[2][]{%
    \errmessage{(Inkscape) Color is used for the text in Inkscape, but the package 'color.sty' is not loaded}%
    \renewcommand\color[2][]{}%
  }%
  \providecommand\transparent[1]{%
    \errmessage{(Inkscape) Transparency is used (non-zero) for the text in Inkscape, but the package 'transparent.sty' is not loaded}%
    \renewcommand\transparent[1]{}%
  }%
  \providecommand\rotatebox[2]{#2}%
  \newcommand*\fsize{\dimexpr\f@size pt\relax}%
  \newcommand*\lineheight[1]{\fontsize{\fsize}{#1\fsize}\selectfont}%
  \ifx\svgwidth\undefined%
    \setlength{\unitlength}{227.59000397bp}%
    \ifx\svgscale\undefined%
      \relax%
    \else%
      \setlength{\unitlength}{\unitlength * \real{\svgscale}}%
    \fi%
  \else%
    \setlength{\unitlength}{\svgwidth}%
  \fi%
  \global\let\svgwidth\undefined%
  \global\let\svgscale\undefined%
  \makeatother%
  \begin{picture}(1,1.02335029)%
    \lineheight{1}%
    \setlength\tabcolsep{0pt}%
    \put(0,0){\includegraphics[width=\unitlength,page=1]{Fig3-b.pdf}}%
    \put(0.05325665,0.98277568){\makebox(0,0)[lt]{\lineheight{0}\smash{\begin{tabular}[t]{l}{\LARGE $\frac{a\phantom{_1}}{a_1}>0$}\end{tabular}}}}%
    \put(0.4837182,0.96260423){\makebox(0,0)[lt]{\lineheight{0}\smash{\begin{tabular}[t]{l}{\large $a_1$}\end{tabular}}}}%
    \put(0.47444989,0.65283668){\makebox(0,0)[lt]{\lineheight{0}\smash{\begin{tabular}[t]{l}{\large $a$}\end{tabular}}}}%
    \put(0.06117118,0.65919577){\makebox(0,0)[lt]{\lineheight{0}\smash{\begin{tabular}[t]{l}{\LARGE $\Omega_1$}\end{tabular}}}}%
    \put(0.06776198,0.23408923){\makebox(0,0)[lt]{\lineheight{0}\smash{\begin{tabular}[t]{l}{\LARGE $\Omega$}\end{tabular}}}}%
  \end{picture}%
\endgroup%

%% file: Fig3-a.pdf_tex
\begingroup%
  \makeatletter%
  \providecommand\color[2][]{%
    \errmessage{(Inkscape) Color is used for the text in Inkscape, but the package 'color.sty' is not loaded}%
    \renewcommand\color[2][]{}%
  }%
  \providecommand\transparent[1]{%
    \errmessage{(Inkscape) Transparency is used (non-zero) for the text in Inkscape, but the package 'transparent.sty' is not loaded}%
    \renewcommand\transparent[1]{}%
  }%
  \providecommand\rotatebox[2]{#2}%
  \newcommand*\fsize{\dimexpr\f@size pt\relax}%
  \newcommand*\lineheight[1]{\fontsize{\fsize}{#1\fsize}\selectfont}%
  \ifx\svgwidth\undefined%
    \setlength{\unitlength}{230.05817413bp}%
    \ifx\svgscale\undefined%
      \relax%
    \else%
      \setlength{\unitlength}{\unitlength * \real{\svgscale}}%
    \fi%
  \else%
    \setlength{\unitlength}{\svgwidth}%
  \fi%
  \global\let\svgwidth\undefined%
  \global\let\svgscale\undefined%
  \makeatother%
  \begin{picture}(1,0.99295889)%
    \lineheight{1}%
    \setlength\tabcolsep{0pt}%
    \put(0,0){\includegraphics[width=\unitlength,page=1]{Fig3-a.pdf}}%
    \put(0.46392815,0.75721688){\makebox(0,0)[lt]{\lineheight{0}\smash{\begin{tabular}[t]{l}{\large $a_1$}\end{tabular}}}}%
    \put(0.47044823,0.64311531){\makebox(0,0)[lt]{\lineheight{0}\smash{\begin{tabular}[t]{l}{\large $a$}\end{tabular}}}}%
    \put(0.20964464,0.96585976){\makebox(0,0)[lt]{\lineheight{0}\smash{\begin{tabular}[t]{l}{\LARGE $\frac{a\phantom{_1}}{a_1}<0$}\end{tabular}}}}%
    \put(0.08985497,0.67835181){\makebox(0,0)[lt]{\lineheight{0}\smash{\begin{tabular}[t]{l}{\LARGE $\Omega_1$}\end{tabular}}}}%
    \put(0.08659493,0.22846561){\makebox(0,0)[lt]{\lineheight{0}\smash{\begin{tabular}[t]{l}{\LARGE $\Omega$}\end{tabular}}}}%
  \end{picture}%
\endgroup%

%% file: PP-Upper-1.pdf_tex
\begingroup%
  \makeatletter%
  \providecommand\color[2][]{%
    \errmessage{(Inkscape) Color is used for the text in Inkscape, but the package 'color.sty' is not loaded}%
    \renewcommand\color[2][]{}%
  }%
  \providecommand\transparent[1]{%
    \errmessage{(Inkscape) Transparency is used (non-zero) for the text in Inkscape, but the package 'transparent.sty' is not loaded}%
    \renewcommand\transparent[1]{}%
  }%
  \providecommand\rotatebox[2]{#2}%
  \newcommand*\fsize{\dimexpr\f@size pt\relax}%
  \newcommand*\lineheight[1]{\fontsize{\fsize}{#1\fsize}\selectfont}%
  \ifx\svgwidth\undefined%
    \setlength{\unitlength}{491.11577072bp}%
    \ifx\svgscale\undefined%
      \relax%
    \else%
      \setlength{\unitlength}{\unitlength * \real{\svgscale}}%
    \fi%
  \else%
    \setlength{\unitlength}{\svgwidth}%
  \fi%
  \global\let\svgwidth\undefined%
  \global\let\svgscale\undefined%
  \makeatother%
  \begin{picture}(1,0.70809725)%
    \lineheight{1}%
    \setlength\tabcolsep{0pt}%
    \put(0,0){\includegraphics[width=\unitlength,page=1]{PP-Upper-1.pdf}}%
    \put(0.80575409,0.06937799){\color[rgb]{1,1,1}\makebox(0,0)[lt]{\lineheight{0}\smash{\begin{tabular}[t]{l}\Large{\textcolor{red}{Separatrix}}\end{tabular}}}}%
    \put(0.04566466,0.13122108){\makebox(0,0)[lt]{\lineheight{0}\smash{\begin{tabular}[t]{l}\Large{\textcolor{gray}{$\infty-$isocline}}\end{tabular}}}}%
    \put(0,0){\includegraphics[width=\unitlength,page=2]{PP-Upper-1.pdf}}%
    \put(0.96255133,0.38919872){\makebox(0,0)[lt]{\lineheight{0}\smash{\begin{tabular}[t]{l}\Large{${(y=h_1)}$}\end{tabular}}}}%
    \put(0.51869461,0.19876837){\makebox(0,0)[lt]{\lineheight{0}\smash{\begin{tabular}[t]{l}\Large{$\textcolor{red}{Q_1=(0,Y_1^*)}$}\end{tabular}}}}%
    \put(0,0){\includegraphics[width=\unitlength,page=3]{PP-Upper-1.pdf}}%
    \put(-0.08815797,0.39013655){\makebox(0,0)[lt]{\lineheight{0}\smash{\begin{tabular}[t]{l}\textcolor{azure}{\Large{$Y_1=0$}}\end{tabular}}}}%
    \put(0.96176177,0.35195825){\makebox(0,0)[lt]{\lineheight{0}\smash{\begin{tabular}[t]{l}\Large{$(y=h_1+a_1)$}\end{tabular}}}}%
    \put(0,0){\includegraphics[width=\unitlength,page=4]{PP-Upper-1.pdf}}%
    \put(-0.11853588,0.35571458){\makebox(0,0)[lt]{\lineheight{0}\smash{\begin{tabular}[t]{l}\textcolor{azure}{\Large{$Y_1=-\mathfrak e_1$}}\end{tabular}}}}%
    \put(0,0){\includegraphics[width=\unitlength,page=5]{PP-Upper-1.pdf}}%
    \put(0.96467623,0.51403779){\makebox(0,0)[lt]{\lineheight{0}\smash{\begin{tabular}[t]{l}\Large{$(y=0)$}\end{tabular}}}}%
    \put(0.96650142,0.54057348){\makebox(0,0)[lt]{\lineheight{0}\smash{\begin{tabular}[t]{l}\Large{$(y=-a)$}\end{tabular}}}}%
    \put(-0.12125119,0.51599654){\makebox(0,0)[lt]{\lineheight{0}\smash{\begin{tabular}[t]{l}\textcolor{bleudefrance}{\Large{$Y_1=kh_1$}}\end{tabular}}}}%
    \put(-0.1765559,0.54599828){\makebox(0,0)[lt]{\lineheight{0}\smash{\begin{tabular}[t]{l}\textcolor{bleudefrance}{\Large{$Y_1=kh_1+\mathfrak e$}}\end{tabular}}}}%
    \put(0,0){\includegraphics[width=\unitlength,page=6]{PP-Upper-1.pdf}}%
    \put(0.46061605,0.0002643){\makebox(0,0)[lt]{\lineheight{0}\smash{\begin{tabular}[t]{l}\Large{$X=0$}\end{tabular}}}}%
    \put(0.68052372,0.0002643){\makebox(0,0)[lt]{\lineheight{0}\smash{\begin{tabular}[t]{l}\Large{$X=\frac{\pi}{2}$}\end{tabular}}}}%
    \put(0.92791945,-0.00126281){\makebox(0,0)[lt]{\lineheight{0}\smash{\begin{tabular}[t]{l}\Large{$X=\pi$}\end{tabular}}}}%
    \put(0.21474752,0.00026435){\makebox(0,0)[lt]{\lineheight{0}\smash{\begin{tabular}[t]{l}\Large{$X=-\frac{\pi}{2}$}\end{tabular}}}}%
    \put(0.01011141,0.00179142){\makebox(0,0)[lt]{\lineheight{0}\smash{\begin{tabular}[t]{l}\Large{$X=-\pi$}\end{tabular}}}}%
    \put(0,0){\includegraphics[width=\unitlength,page=7]{PP-Upper-1.pdf}}%
    \put(-0.09265987,0.42786205){\makebox(0,0)[lt]{\lineheight{0}\smash{\begin{tabular}[t]{l}\textcolor{azure}{\Large{$Y_1=\mathfrak e_1$}}\end{tabular}}}}%
    \put(0.96304356,0.4217535){\makebox(0,0)[lt]{\lineheight{0}\smash{\begin{tabular}[t]{l}\Large{$(y=h_1-a_1)$}\end{tabular}}}}%
  \end{picture}%
\endgroup%

%% file: fY1a.pdf_tex
\begingroup%
  \makeatletter%
  \providecommand\color[2][]{%
    \errmessage{(Inkscape) Color is used for the text in Inkscape, but the package 'color.sty' is not loaded}%
    \renewcommand\color[2][]{}%
  }%
  \providecommand\transparent[1]{%
    \errmessage{(Inkscape) Transparency is used (non-zero) for the text in Inkscape, but the package 'transparent.sty' is not loaded}%
    \renewcommand\transparent[1]{}%
  }%
  \providecommand\rotatebox[2]{#2}%
  \newcommand*\fsize{\dimexpr\f@size pt\relax}%
  \newcommand*\lineheight[1]{\fontsize{\fsize}{#1\fsize}\selectfont}%
  \ifx\svgwidth\undefined%
    \setlength{\unitlength}{502.39578848bp}%
    \ifx\svgscale\undefined%
      \relax%
    \else%
      \setlength{\unitlength}{\unitlength * \real{\svgscale}}%
    \fi%
  \else%
    \setlength{\unitlength}{\svgwidth}%
  \fi%
  \global\let\svgwidth\undefined%
  \global\let\svgscale\undefined%
  \makeatother%
  \begin{picture}(1,0.49631819)%
    \lineheight{1}%
    \setlength\tabcolsep{0pt}%
    \put(0,0){\includegraphics[width=\unitlength,page=1]{fY1a.pdf}}%
    \put(0.97417876,0.00318533){\makebox(0,0)[lt]{\lineheight{0}\smash{\begin{tabular}[t]{l}\Large{$Y_1$}\end{tabular}}}}%
    \put(0.97600972,0.3855926){\makebox(0,0)[lt]{\lineheight{0}\smash{\begin{tabular}[t]{l}\Large{$f(Y_1)$}\end{tabular}}}}%
    \put(0.24830572,0.19893561){\makebox(0,0)[lt]{\lineheight{0}\smash{\begin{tabular}[t]{l}\Large{$f(-\mathfrak e_1)\approx A+\mathfrak e_1$}\end{tabular}}}}%
    \put(0.54206555,0.00404687){\makebox(0,0)[lt]{\lineheight{0}\smash{\begin{tabular}[t]{l}\Large{$\bar{Y_1}$}\end{tabular}}}}%
    \put(0.28861871,0.00404102){\makebox(0,0)[lt]{\lineheight{0}\smash{\begin{tabular}[t]{l}\Large{$0$}\end{tabular}}}}%
    \put(0,0){\includegraphics[width=\unitlength,page=2]{fY1a.pdf}}%
    \put(0.22797378,0.0050967){\makebox(0,0)[lt]{\lineheight{0}\smash{\begin{tabular}[t]{l}\Large{$-\mathfrak e_1$}\end{tabular}}}}%
    \put(0,0){\includegraphics[width=\unitlength,page=3]{fY1a.pdf}}%
    \put(0.90221641,0.0044194){\makebox(0,0)[lt]{\lineheight{0}\smash{\begin{tabular}[t]{l}\Large{${Y^*_1}$}\end{tabular}}}}%
    \put(0.0490851,0.00404687){\makebox(0,0)[lt]{\lineheight{0}\smash{\begin{tabular}[t]{l}\Large{$\tilde{Y^*_1}$}\end{tabular}}}}%
    \put(0.30757328,0.32356076){\makebox(0,0)[lt]{\lineheight{0}\smash{\begin{tabular}[t]{l}\Large{$1/\mathfrak e_1=f(\tilde{Y^*_1})=f(Y_1^*)$}\end{tabular}}}}%
    \put(0,0){\includegraphics[width=\unitlength,page=4]{fY1a.pdf}}%
    \put(0.74952858,0.00614145){\color[rgb]{0.12941176,0.67058824,0.80392157}\makebox(0,0)[lt]{\lineheight{0}\smash{\begin{tabular}[t]{l}\Large{$\#2:kh_1$}\end{tabular}}}}%
    \put(0.7504689,0.24258372){\color[rgb]{0.12941176,0.67058824,0.80392157}\makebox(0,0)[lt]{\lineheight{0}\smash{\begin{tabular}[t]{l}\Large{$f(kh_1)$}\end{tabular}}}}%
    \put(0,0){\includegraphics[width=\unitlength,page=5]{fY1a.pdf}}%
    \put(0.45001068,0.15221274){\color[rgb]{0,0.54117647,1}\makebox(0,0)[lt]{\lineheight{0}\smash{\begin{tabular}[t]{l}\Large{$f(kh_1)$}\end{tabular}}}}%
    \put(0.41887732,0.00507931){\color[rgb]{0.19215686,0.54901961,0.90588235}\makebox(0,0)[lt]{\lineheight{0}\smash{\begin{tabular}[t]{l}\Large{$\#1:kh_1$}\end{tabular}}}}%
    \put(0,0){\includegraphics[width=\unitlength,page=6]{fY1a.pdf}}%
  \end{picture}%
\endgroup%

%% file: PP-Upper.pdf_tex
\begingroup%
  \makeatletter%
  \providecommand\color[2][]{%
    \errmessage{(Inkscape) Color is used for the text in Inkscape, but the package 'color.sty' is not loaded}%
    \renewcommand\color[2][]{}%
  }%
  \providecommand\transparent[1]{%
    \errmessage{(Inkscape) Transparency is used (non-zero) for the text in Inkscape, but the package 'transparent.sty' is not loaded}%
    \renewcommand\transparent[1]{}%
  }%
  \providecommand\rotatebox[2]{#2}%
  \newcommand*\fsize{\dimexpr\f@size pt\relax}%
  \newcommand*\lineheight[1]{\fontsize{\fsize}{#1\fsize}\selectfont}%
  \ifx\svgwidth\undefined%
    \setlength{\unitlength}{483.46971887bp}%
    \ifx\svgscale\undefined%
      \relax%
    \else%
      \setlength{\unitlength}{\unitlength * \real{\svgscale}}%
    \fi%
  \else%
    \setlength{\unitlength}{\svgwidth}%
  \fi%
  \global\let\svgwidth\undefined%
  \global\let\svgscale\undefined%
  \makeatother%
  \begin{picture}(1,1.0646653)%
    \lineheight{1}%
    \setlength\tabcolsep{0pt}%
    \put(0,0){\includegraphics[width=\unitlength,page=1]{PP-Upper.pdf}}%
    \put(0.92635643,0.00357529){\makebox(0,0)[lt]{\lineheight{0}\smash{\begin{tabular}[t]{l}\Large{$X=\pi$}\end{tabular}}}}%
    \put(0.22813461,0.00546477){\makebox(0,0)[lt]{\lineheight{0}\smash{\begin{tabular}[t]{l}\Large{$X=-\frac{\pi}{2}$}\end{tabular}}}}%
    \put(-0.00084836,0.00576171){\makebox(0,0)[lt]{\lineheight{0}\smash{\begin{tabular}[t]{l}\Large{$X=-\pi$}\end{tabular}}}}%
    \put(0.70297115,0.0020238){\makebox(0,0)[lt]{\lineheight{0}\smash{\begin{tabular}[t]{l}\Large{$X=\frac{\pi}{2}$}\end{tabular}}}}%
    \put(0.45011147,0.00667773){\makebox(0,0)[lt]{\lineheight{0}\smash{\begin{tabular}[t]{l}\Large {$X=0$}\end{tabular}}}}%
    \put(-0.11110583,0.48519459){\makebox(0,0)[lt]{\lineheight{0}\smash{\begin{tabular}[t]{l}\Large{${Y_1=\bar {Y_1}}$}\end{tabular}}}}%
    \put(0.77371448,0.81490935){\makebox(0,0)[lt]{\lineheight{0}\smash{\begin{tabular}[t]{l}\Large{\textcolor{red}{Separatrix}}\end{tabular}}}}%
    \put(0.03206706,0.92367503){\makebox(0,0)[lt]{\lineheight{0}\smash{\begin{tabular}[t]{l}\Large{\textcolor{gray}{$\infty-$isocline}}\end{tabular}}}}%
    \put(0,0){\includegraphics[width=\unitlength,page=2]{PP-Upper.pdf}}%
    \put(-0.09340202,0.32084454){\makebox(0,0)[lt]{\lineheight{0}\smash{\begin{tabular}[t]{l}\textcolor{azure}{\Large{$Y_1=0$}}\\\end{tabular}}}}%
    \put(0,0){\includegraphics[width=\unitlength,page=3]{PP-Upper.pdf}}%
    \put(0.52446835,0.78022003){\makebox(0,0)[lt]{\lineheight{0}\smash{\begin{tabular}[t]{l}\Large{$\textcolor{red}{Q_1=(0,Y_1^*)}$}\end{tabular}}}}%
    \put(-0.11645334,0.2789597){\makebox(0,0)[lt]{\lineheight{0}\smash{\begin{tabular}[t]{l}\textcolor{azure}{\Large{$Y_1=-\mathfrak e_1$}}\end{tabular}}}}%
    \put(0.9797487,0.27895988){\makebox(0,0)[lt]{\lineheight{0}\smash{\begin{tabular}[t]{l}\Large{$(y=h_1+a_1)$}\end{tabular}}}}%
    \put(0.97654919,0.32084454){\makebox(0,0)[lt]{\lineheight{0}\smash{\begin{tabular}[t]{l}\Large{$(y=h_1)$}\end{tabular}}}}%
    \put(0,0){\includegraphics[width=\unitlength,page=4]{PP-Upper.pdf}}%
    \put(-0.09374154,0.3642806){\makebox(0,0)[lt]{\lineheight{0}\smash{\begin{tabular}[t]{l}\textcolor{azure}{\Large{$Y_1=\mathfrak e_1$}}\end{tabular}}}}%
    \put(0.97810037,0.36428069){\makebox(0,0)[lt]{\lineheight{0}\smash{\begin{tabular}[t]{l}\Large{$(y=h_1-a_1)$}\end{tabular}}}}%
    \put(0,0){\includegraphics[width=\unitlength,page=5]{PP-Upper.pdf}}%
    \put(0.5313785,0.11142079){\makebox(0,0)[lt]{\lineheight{0}\smash{\begin{tabular}[t]{l}\Large{$\textcolor{red}{\tilde{Q_1}=(0,\tilde{Y_1^*})}$}\end{tabular}}}}%
    \put(0,0){\includegraphics[width=\unitlength,page=6]{PP-Upper.pdf}}%
    \put(-0.27869006,0.42320239){\makebox(0,0)[lt]{\lineheight{0}\smash{\begin{tabular}[t]{l}\textcolor{bleudefrance}{\Large{\#1:} \Large{$Y_1=kh_1\leq\bar{Y_1}$}}\end{tabular}}}}%
    \put(-0.27988369,0.55506161){\makebox(0,0)[lt]{\lineheight{0}\smash{\begin{tabular}[t]{l}\textcolor{ballblue}{\Large{\#2:} \Large{$Y_1=kh_1>\bar{Y_1}$}}\end{tabular}}}}%
    \put(0.97465255,0.42165099){\makebox(0,0)[lt]{\lineheight{0}\smash{\begin{tabular}[t]{l}\Large{$(y=0)$}\end{tabular}}}}%
    \put(0.97620386,0.55661301){\makebox(0,0)[lt]{\lineheight{0}\smash{\begin{tabular}[t]{l}\Large{$(y=0)$}\end{tabular}}}}%
    \put(0,0){\includegraphics[width=\unitlength,page=7]{PP-Upper.pdf}}%
    \put(0.7909321,0.11294499){\makebox(0,0)[lt]{\lineheight{0}\smash{\begin{tabular}[t]{l}\Large{\textcolor{gray}{$\infty-$isocline}}\end{tabular}}}}%
    \put(0,0){\includegraphics[width=\unitlength,page=8]{PP-Upper.pdf}}%
    \put(0.0650512,0.09134823){\makebox(0,0)[lt]{\lineheight{0}\smash{\begin{tabular}[t]{l}\Large{\textcolor{red}{Separatrix}}\end{tabular}}}}%
  \end{picture}%
\endgroup%

%% file: PartLower.pdf_tex
\begingroup%
  \makeatletter%
  \providecommand\color[2][]{%
    \errmessage{(Inkscape) Color is used for the text in Inkscape, but the package 'color.sty' is not loaded}%
    \renewcommand\color[2][]{}%
  }%
  \providecommand\transparent[1]{%
    \errmessage{(Inkscape) Transparency is used (non-zero) for the text in Inkscape, but the package 'transparent.sty' is not loaded}%
    \renewcommand\transparent[1]{}%
  }%
  \providecommand\rotatebox[2]{#2}%
  \newcommand*\fsize{\dimexpr\f@size pt\relax}%
  \newcommand*\lineheight[1]{\fontsize{\fsize}{#1\fsize}\selectfont}%
  \ifx\svgwidth\undefined%
    \setlength{\unitlength}{160.83930588bp}%
    \ifx\svgscale\undefined%
      \relax%
    \else%
      \setlength{\unitlength}{\unitlength * \real{\svgscale}}%
    \fi%
  \else%
    \setlength{\unitlength}{\svgwidth}%
  \fi%
  \global\let\svgwidth\undefined%
  \global\let\svgscale\undefined%
  \makeatother%
  \begin{picture}(1,0.63455948)%
    \lineheight{1}%
    \setlength\tabcolsep{0pt}%
    \put(0,0){\includegraphics[width=\unitlength,page=1]{PartLower.pdf}}%
    \put(0.36693295,0.01471217){\makebox(0,0)[lt]{\lineheight{0}\smash{\begin{tabular}[t]{l}$\Large{A}$\end{tabular}}}}%
    \put(0.44295143,0.01596661){\color[rgb]{1,0,0}\makebox(0,0)[lt]{\lineheight{0}\smash{\begin{tabular}[t]{l}$\Large{B}$\end{tabular}}}}%
    \put(0.52043736,0.01514237){\color[rgb]{0,0,1}\makebox(0,0)[lt]{\lineheight{0}\smash{\begin{tabular}[t]{l}$\Large{C}$\end{tabular}}}}%
    \put(0.59709844,0.01679111){\makebox(0,0)[lt]{\lineheight{0}\smash{\begin{tabular}[t]{l}$\Large{D}$\end{tabular}}}}%
  \end{picture}%
\endgroup%

%% file: PartUp1.pdf_tex
\begingroup%
  \makeatletter%
  \providecommand\color[2][]{%
    \errmessage{(Inkscape) Color is used for the text in Inkscape, but the package 'color.sty' is not loaded}%
    \renewcommand\color[2][]{}%
  }%
  \providecommand\transparent[1]{%
    \errmessage{(Inkscape) Transparency is used (non-zero) for the text in Inkscape, but the package 'transparent.sty' is not loaded}%
    \renewcommand\transparent[1]{}%
  }%
  \providecommand\rotatebox[2]{#2}%
  \newcommand*\fsize{\dimexpr\f@size pt\relax}%
  \newcommand*\lineheight[1]{\fontsize{\fsize}{#1\fsize}\selectfont}%
  \ifx\svgwidth\undefined%
    \setlength{\unitlength}{160.91408157bp}%
    \ifx\svgscale\undefined%
      \relax%
    \else%
      \setlength{\unitlength}{\unitlength * \real{\svgscale}}%
    \fi%
  \else%
    \setlength{\unitlength}{\svgwidth}%
  \fi%
  \global\let\svgwidth\undefined%
  \global\let\svgscale\undefined%
  \makeatother%
  \begin{picture}(1,1.30319673)%
    \lineheight{1}%
    \setlength\tabcolsep{0pt}%
    \put(0,0){\includegraphics[width=\unitlength,page=1]{PartUp1.pdf}}%
    \put(0.34243498,0.91861117){\makebox(0,0)[lt]{\lineheight{0}\smash{\begin{tabular}[t]{l}$\Large{A}$\end{tabular}}}}%
    \put(0.41841813,0.91986503){\color[rgb]{1,0,0}\makebox(0,0)[lt]{\lineheight{0}\smash{\begin{tabular}[t]{l}$\Large{B}$\end{tabular}}}}%
    \put(0.49586806,0.91904118){\color[rgb]{0,0,1}\makebox(0,0)[lt]{\lineheight{0}\smash{\begin{tabular}[t]{l}$\Large{C}$\end{tabular}}}}%
    \put(0.57249352,0.92068914){\makebox(0,0)[lt]{\lineheight{0}\smash{\begin{tabular}[t]{l}$\Large{D}$\end{tabular}}}}%
    \put(0,0){\includegraphics[width=\unitlength,page=2]{PartUp1.pdf}}%
    \put(0.36235394,0.45038632){\makebox(0,0)[lt]{\lineheight{0}\smash{\begin{tabular}[t]{l}$\Large{A}$\end{tabular}}}}%
    \put(0.43833764,0.45164018){\color[rgb]{1,0,0}\makebox(0,0)[lt]{\lineheight{0}\smash{\begin{tabular}[t]{l}$\Large{B}$\end{tabular}}}}%
    \put(0.51578757,0.45081633){\color[rgb]{0,0,1}\makebox(0,0)[lt]{\lineheight{0}\smash{\begin{tabular}[t]{l}$\Large{C}$\end{tabular}}}}%
    \put(0.59241247,0.4524643){\makebox(0,0)[lt]{\lineheight{0}\smash{\begin{tabular}[t]{l}$\Large{D}$\end{tabular}}}}%
    \put(0,0){\includegraphics[width=\unitlength,page=3]{PartUp1.pdf}}%
    \put(0.35003006,0.68856013){\makebox(0,0)[lt]{\lineheight{0}\smash{\begin{tabular}[t]{l}$\Large{A}$\end{tabular}}}}%
    \put(0.42601321,0.68981399){\color[rgb]{1,0,0}\makebox(0,0)[lt]{\lineheight{0}\smash{\begin{tabular}[t]{l}$\Large{B}$\end{tabular}}}}%
    \put(0.50346314,0.68899014){\color[rgb]{0,0,1}\makebox(0,0)[lt]{\lineheight{0}\smash{\begin{tabular}[t]{l}$\Large{C}$\end{tabular}}}}%
    \put(0.5800886,0.6906381){\makebox(0,0)[lt]{\lineheight{0}\smash{\begin{tabular}[t]{l}$\Large{D}$\end{tabular}}}}%
    \put(0,0){\includegraphics[width=\unitlength,page=4]{PartUp1.pdf}}%
    \put(-0.09134065,1.12142265){\makebox(0,0)[lt]{\lineheight{0}\smash{\begin{tabular}[t]{l}$\Large{(a)}$\end{tabular}}}}%
    \put(-0.09510907,0.76018672){\makebox(0,0)[lt]{\lineheight{0}\smash{\begin{tabular}[t]{l}$\Large{(b)}$\end{tabular}}}}%
    \put(-0.09977004,0.20321272){\makebox(0,0)[lt]{\lineheight{0}\smash{\begin{tabular}[t]{l}$\Large{(c)}$\end{tabular}}}}%
  \end{picture}%
\endgroup%